\RequirePackage{fix-cm}
\RequirePackage[hyphens]{url}%
\documentclass[smallcondensed]{svjour3hack}     

\smartqed  
\usepackage{graphicx}
\usepackage{placeins}
\usepackage{amsmath}
\usepackage{amssymb}
\usepackage{mathtools} 
\usepackage[dvipsnames]{xcolor}

\makeatletter
\renewcommand*{\top}{%
  {\mathpalette\@transpose{}}%
}
\newcommand*{\@transpose}[2]{%
  \raisebox{\depth}{$\m@th#1\scriptscriptstyle\mathsf{T}$}%
}
\makeatother

\newcommand\jatop[2]{\genfrac{}{}{0pt}{}{#1\hfill}{#2\hfill}}
\newcommand\katop[2]{\genfrac{}{}{0pt}{}{#1}{#2}}

\DeclareMathOperator{\rank}{rank}

\DeclareMathOperator*{\argmin}{argmin}

\DeclareMathOperator{\vvec}{\textstyle{vec}}
\DeclareMathOperator{\tr}{\textstyle{tr}}
\DeclareMathOperator{\prox}{\textstyle{prox}}

\makeatletter
\let\c@proposition\c@theorem
\let\c@corollary\c@theorem
\let\c@lemma\c@theorem
\let\c@definition\c@theorem
\let\c@example\c@theorem
\let\c@remark\c@theorem
\let\c@observation\c@theorem
\let\c@claim\c@theorem
\makeatother

\usepackage[hypertexnames=false,colorlinks=true,breaklinks=true,bookmarks=true,urlcolor=blue,citecolor=blue,linkcolor=blue,bookmarksopen=false,draft=false]{hyperref}

\allowdisplaybreaks[4]

\journalname{Optimization and Engineering}

\begin{document}

\title{On computing  sparse universal solvers\\ for key problems in statistics}

\titlerunning{Sparse universal  solvers for key problems in statistics}        

\author{Ananias Sousa Machado \break \and
        Marcia Fampa \and Jon Lee  
}

\authorrunning{A.S. Machado, M. Fampa \& J. Lee} 

\institute{A.S. Machado \at
              Universidade Federal do Rio de Janeiro  \\
              \email{ananiasmachado@poli.ufrj.br}           
           \and
           M. Fampa \at
              Universidade Federal do Rio de Janeiro \\
              \email{fampa@cos.ufrj.br}           
        \and
           J. Lee \at
           University of Michigan\\
           \email{jonxlee@umich.edu}
}

\date{September 5, 2025. Revised June 16, 2026.}

\maketitle

\begin{abstract}
We give sparsity results and present algorithms for calculating minimum (vector) 1-norm \emph{universal solvers} connected to
least-squares problems. In particular, besides 
universal least-squares solvers, we consider 
\emph{minimum-rank} universal least-squares solvers, and simultaneous universal minimum-norm/least-squares solvers. For all of these, we present and compare several new alternative linear-programming formulations and very effective proximal-point algorithms. Overall, we found that our new Douglas–Rachford splitting algorithms
for these problems performed best. 
\end{abstract}

\keywords{least squares, universal solver, Moore-Penrose properties, generalized inverse, minimum 1-norm, sparsity, linear programming, ADMM, Douglas–Rachford splitting algorithm}

\section{Introduction}

A \emph{universal solver} for a problem defined by pairs $(A\in\mathbb{R}^{m\times n},b\in\mathbb{R}^m)$, where $A$ is considered fixed and $b$ varies,
is a ``generalized inverse''  
$H\in\mathbb{R}^{n\times m}$ of $A$ that produces an optimal solution of the problem, via $\hat\theta:=Hb$, for all $b$ in some prescribed set. For some key problems of
interest, $H:=A^\dagger$, the Moore-Penrose pseudoinverse of $A$, is a universal solver. For example: (i)
the least-squares problem
$\min_\theta \|A\theta -b\|^2_2$\,, for arbitrary $b\in\mathbb{R}^m$, and (ii)
the minimum 2-norm problem 
$\min_\theta \{\|\theta\|^2_2 : A\theta=b \}$, when $b$ is in the column space 
of $A$. 
But typically, $A^\dagger$ is rather dense, even when $A$ is sparse, and so computing the solution $\hat \theta := Hb$ for many vectors $b$ is expensive.  So we are interested in finding \emph{sparse} universal solvers, when
$A$ has neither full-row nor full-column rank (which is the situation that can lead to sparser universal solvers than $A^\dagger$). 
Sparsity assumptions on $A$ are not the focus of our work, however we note that many calculations 
for the various systems that we work with may take advantage of sparsity or structured sparsity on $A$.

Our starting point is the following celebrated 
characterization of the ``Moore-Penrose pseudoinverse''.
\begin{theorem}[see \protect{\cite[Thm. 1]{Penrose}}]
For $A\!\in\!\mathbb{R}^{m \times n}$, the M-P pseudoinverse $A^{\dagger}$ is the unique 
 $H\!\in\!\mathbb{R}^{n \times m}$ satisfying the following conditions.
	\begin{align}
\mbox{(generalized inverse)\quad	 }	& AHA = A \label{P1} \tag{P1}\\
\mbox{(reflexive)\quad }		& HAH = H \label{P2} \tag{P2}\\
\mbox{(ah-symmetric)\quad }		& (AH)^{\top} = AH \label{P3} \tag{P3}\\
\mbox{(ha-symmetric)\quad }		& (HA)^{\top} = HA. \label{P4} \tag{P4}
	\end{align}
\end{theorem}
The terms ``generalized inverse'' and ``reflexive'' are completely standard.
We have been using the terms ``ah-symmetric'' and ``ha-symmetric'' in our previous works as convenient mnemonics for the symmetry of $AH$ and $HA$, respectively. Mixing, we also say ``ah-ha-symmetric''
when both $AH$ and $HA$ are symmetric. 

Some key universal solvers $H$ are characterized by proper subsets of the M-P properties, which leaves room for searching for sparse universal solvers
using optimization methods. This approach was initiated in \cite{FFL2016}
and further developed in many works; see the recent work \cite{ponte2024goodfastrowsparseahsymmetric} and the references therein.
Generally, we employ the common technique of seeking to 
induce high sparsity (low ``0-norm'') by minimizing 
the vector 1-norm. Additionally, minimizing the vector 1-norm --- a true norm --- has the side benefit of 
keeping the magnitudes of the entries under control, which leads to solutions with better numerical properties. 

 For rank $r$ matrix $A\in\mathbb{R}^{m\times n}$,
for the \emph{(full) singular-value decomposition}  (SVD), we  write $A=U\Sigma V^\top$, where  $U_{m\times m}$ and $V_{n\times n}$ are real orthogonal matrices, 
 $\Sigma_{m\times n}$ is a diagonal matrix, with non-negative diagonal entries
  $\sigma_1\geq\sigma_2\geq \cdots \geq \sigma_{\min\{m,n\}}$\,, known as the  singular values of $A$. 
The following very useful theorem characterizes each of the M-P properties via the
the SVD of $A$.

\begin{theorem}[\cite{BenIsrael1974}, p. 208, Ex. 14 (no~proof); also see \cite{PFLX_ORL} (with proof)]\label{thm:structural}
For rank $r$ matrix $A\in\mathbb{R}^{m\times n}$, consider the full SVD $A =: U \Sigma V^\top$ with 
\[
\Sigma =: \begin{bmatrix}\underset{\scriptscriptstyle r\times r}{D} & \underset{\scriptscriptstyle r\times (n-r)}{0}\\
\underset{\scriptscriptstyle (m-r)\times r}{0} & \underset{\scriptscriptstyle (m-r)\times (n-r)}{0}\end{bmatrix},
\]
\noindent where $V:= \begin{bmatrix}\underset{\scriptscriptstyle n\times r}{V_1} & \underset{\scriptscriptstyle n\times (n-r)}{V_2}\end{bmatrix}, ~ U:= \begin{bmatrix}\underset{\scriptscriptstyle m\times r}{U_1} & \underset{\scriptscriptstyle m\times (m-r)}{U_2}\end{bmatrix}$, and $D$ is diagonal with positive diagonal components. 
For $H \!\in\! \mathbb{R}^{n \times m}$, let $\Gamma\!:=\!V^\top H U$ (so $H\!=\! V\Gamma U^\top$), where we block partition 
$\Gamma$~as 
\[
\Gamma=:\begin{bmatrix}\underset{\scriptscriptstyle r\times r}{X} & \underset{\scriptscriptstyle r\times (m-r)}{Y}\\
\underset{\scriptscriptstyle (n-r)\times r}{Z} & \underset{\scriptscriptstyle (n-r)\times (m-r)}{W}\end{bmatrix}.
\]
We have 
\begin{enumerate}
\item[($i$)]
{\rm\ref{P1}} is equivalent to  
 $X = D^{-1}$.
\item[($ii$)]
  If {\rm\ref{P1}} is satisfied, then {\rm\ref{P2}} is equivalent to 
  $ZDY = W$.
\item[($iii$)] 
If {\rm\ref{P1}} is satisfied, then {\rm\ref{P3}} is equivalent to 
    $Y = 0$.
\item[($iv$)]
    If {\rm\ref{P1}} is satisfied, then {\rm\ref{P4}} is equivalent to  
    $Z = 0$.
\end{enumerate}
\end{theorem}

\noindent In fact, by setting all of $W$, $Y$ and $Z$ to 0 in Theorem \ref{thm:structural}, we
can see the simple relationship between $A^\dagger$ and the
SVD of $A$.

\medskip

\noindent \textbf{Literature Overview.}
\cite{dokmanic,dokmanic1,dokmanic2} suggested 1-norm minimization for inducing sparsity in left- and right-inverses. 
\cite{FFL2016} introduced the idea of seeking various sparse generalized inverses
based on subsets of the Moore-Penrose properties,
using 1-norm minimization and general-purpose linear-programming software.
\cite{FampaLee2018ORL} gave a combinatorial polynomial-time approximation algorithm for
 1-norm minimization over reflexive generalized inverses.
 \cite{XFLPsiam} extended these ideas to (i) symmetric generalized inverses of symmetric matrices, and (ii) ah-symmetric generalized inverses of general matrices. They also gave hardness results for
 sparsity maximization for various generalized inverses, and they gave sparsity bounds for various 1-norm minimizing generalized inverses calculated by a linear-programming algorithm; also see \cite{XFLrank12} and \cite{FLPXjogo}.
 \cite{PFLX_ORL} used and developed structural results for various generalized inverses,
 aimed at compact optimization formulations for 1-norm minimization (seeking various types of sparse generalized inverses) and 2,1-norm minimization (seeking various types of row-sparse generalized inverses). \cite{ponte2024goodfastrowsparseahsymmetric} introduced successful ADMM (Alternating Directions Method of Multipliers) algorithms for 
 1-norm and 2,1-norm minimization for ah-symmetric reflexive generalized inverses, and
 they demonstrated that the earlier approximation algorithm for
 1-norm minimization over reflexive generalized inverses also gives an approximation guarantee
 for the 2,1-norm. A few other references on our topic, less relevant to our development, are: \cite{FFL2019,ojmoFLP21,ahsymginv}. 
 Finally, we would like to mention that there is a vast literature on
 finding a sparse solution to an underdetermined system $Ax=b$, when both $A$ and $b$ are considered as fixed input (see \cite{Chen1,Chen2,Chen3}, for example) --- different from our situation where
 we seek a universal solver, which only depends on $A$.

\medskip

\noindent \textbf{Organization and Contributions.} 
In \S\ref{sec:LS}, we concentrate on the problem of 
calculating a minimum 1-norm universal least-squares solver $H$,
together with investigating the sparsity of the resulting $H$. 
We give four different formulations that can be easily reformulated as linear programs (LPs): 
\ref{p131} (based on two of the four M-P properties),
\ref{p1proj13} (based on projection onto the range of $A$),
\ref{ppls1} (based on the least-squares optimality conditions),
and
\ref{calP13} (based on the structure of generalized inverses,
in relation to the M-P properties: Theorem \ref{thm:structural}).
Previously, it was observed that \ref{calP13} is
far superior to \ref{p131}\,, when using a state-of-the-art
commercial linear-programming solver; see \cite{PFLX_ORL}. 
In this context, now we have found that \ref{p1proj13} (which was previously used to prove a sparsity upper bound; see \cite{XFLPsiam}) and \ref{ppls1} are superior to the other two. 
Further, based on \ref{ppls1}\,, we derive a DRS (Douglas-Rachford splitting) algorithm 
which turns out to be far superior to solving any of them using a commercial linear-programming solver. 

In \S\ref{sec:ah-ref}, we look to the problem
of calculating a minimum 1-norm \emph{minimum-rank} universal least-squares solver $H$, together
with investigating the sparsity of the resulting $H$.
We give six different formulations, all but the first of which can be easily reformulated as LPs: 
\ref{p1231a} (based on three of the four M-P properties --- but nonconvex),
\ref{p1231lin} (reformulating the \ref{P2} constraints of \ref{p1231a}, linearly),
\ref{p1proj13p2lin} (based on a projector),
\ref{p1plsp2lin} (based on least-squares optimality conditions),
\ref{pplsr1} (based on a new theorem characterizing feasible solutions of \ref{p1231a}), 
and
\ref{calP123} (based on Theorem \ref{thm:structural}).
It was previously observed that \ref{calP123} is
far superior to \ref{p1231lin}\,, when using a
commercial linear-programming solver; see \cite{PFLX_ORL}. 
We demonstrate that \ref{p1231lin}\,,
\ref{p1proj13p2lin}\,, and 
\ref{p1plsp2lin} are also superior
to \ref{p1231lin}\,, but \ref{calP123}
is better than all of the others.
We note that \ref{p1proj13p2lin} (which performs rather poorly, computationally)
was previously used to prove a sparsity upper bound; see \cite{XFLPsiam}).
All of the resulting LPs are greatly outperformed by an ADMM 
based on 
 \ref{calP123}\,, 
of \cite{ponte2024goodfastrowsparseahsymmetric}. Based on \ref{p1plsp2lin}, we derive a 
new DRS algorithm 
which is far superior to 
ADMM. 

In \S\ref{sec:ah-ha}, we  look at calculating a minimum 1-norm
simultaneous universal mini\-mum-norm/least-squares solver 
$H$, together
with investigating the sparsity of the resulting $H$. 
We give six different formulations, all of which can be easily reformulated as LPs: 
\ref{p1p134} (based on three of the four M-P properties),
\ref{p1pmn3} (based on minimum 2-norm optimality conditions),
\ref{p1plspmn} (based on two sets of optimality conditions),
\ref{double} (based on two projectors),
\ref{p1pmx} (based on a new theorem characterizing feasible solutions of \ref{p1p134}), 
and
\ref{calP134} (based on Theorem \ref{thm:structural}).
The only previous computing was with \ref{p1p134}\,; see \cite{FFL2016}.
Our best results using linear-programming software are with \ref{double}\,. We give (i) a new sparsity upper bound,
based on \ref{p1pmx}\,, (ii) a new ADMM 
based on \ref{calP134}\,, which is vastly superior to all of the LPs, 
and (iii) a new and even-better-performing DRS algorithm, which is based on \ref{p1plspmn}\,.

We have noted that \cite{ponte2024goodfastrowsparseahsymmetric} developed an ADMM
for 
\ref{p1231a} based on \ref{calP123}\,, and using a variable  $Z\in\mathbb{R}^{(n-r) \times r}$\,; see \S\ref{sec:ah-ref}.
Similarly, in  \S\ref{sec:ah-ha}, we develop a new ADMM 
for 
\ref{p1p134} based on \ref{calP134}\,, and using a variable
$W \in \mathbb{R}^{(n-r) \times (m-r)}$\,. We do \emph{not} develop 
an ADMM 
for 
\ref{p131}\, based on \ref{calP13}\,,
because it would require the use of \emph{both} $W$ and $Z$, and we can expect that
it would be quite inefficient (compared to our DRS approach). 

Overall, in \S\S\ref{sec:LS}--\ref{sec:ah-ha},
we find that in our context of seeking minimum 1-norm (and 0-norm inducing)
universal solvers:
\begin{itemize}
    \item There are many viable linear-programming formulations, and with commercial solvers, the performance can vary widely; for small-scale instances, it is
    good to be aware of the different possible formulations. 
    \item The best-performing linear-programming formulations may not be the best ones to analyze for establishing sparsity upper bounds.
    \item Effective ADMM and DRS algorithms, much better performing 
    than applying commercial solvers to  linear-programming formulations,
    may be built from formulations that
    are often not the best for the other purposes (mentioned above).
    \item 1-norm minimization can be very effective for inducing sparsity in generalized inverses if low rank is not required. Enforcing low rank, or equivalently enforcing that \ref{P2} be satisfied, can significantly increase the density of minimal 1-norm ah-symmetric generalized inverses.
\end{itemize}

In \S\ref{sec:outlook}, we give some promising directions for further research.

In Appendix \ref{appA}, we compare \texttt{Gurobi}'s performance on the different linear-programming reformulations presented for the problems addressed in \S\S \ref{sec:LS}, \ref{sec:ah-ref}, and \ref{sec:ah-ha}; and we verify the significant difference on the runtimes for the different formulations. 

\medskip

\noindent \textbf{Notation.} 
For $p\in\{0\}\cup [1,\infty] $,
we denote the ordinary vector $p$-norm by $\|\cdot\|_p$ (of course it is only a pseudo-norm for $p=0$). For a matrix $A\in\mathbb{R}^{m\times n}$, 
we define its (column-wise) vectorization by $\vvec{(A)}\in\mathbb{R}^{mn}$ and 
 its vector $p$-norm by $\|A\|_p:=
\|\vvec(A)\|_p$\,. 
For matrices,
$\tr(\cdot)$ denotes trace, $\det(\cdot)$ denotes determinant, $\rank(\cdot)$ denotes rank,  $\langle \cdot, \cdot \rangle$ 
denotes Frobenius inner product, $\|\cdot\|_F$ denotes Frobenius norm, $\mathcal{R}(\cdot)$ denotes the range (i.e., column space),
and $\mathcal{K}(\cdot)$ denotes the kernel (i.e., null space). 
We denote Hadamard product by $\circ$ and Kronecker product by $\otimes$~.
We have the useful identity: $\vvec(ABC)=(C^\top \otimes A)\vvec(B)$.
We use $\mathbf{e}_i$ to denote the $i$-th standard unit vector, 
$I_m$ denotes an order-$m$ identity matrix, and $J$ denotes an all-ones matrix
(with dimensions inferred from context).
 We denote by $\argmin\{\cdot\}$
\emph{any} solution of the associated minimization problem. 
We say that $H$ satisfies properties $P$+$Q$ (e.g., any of the M-P properties)
if $H$ satisfies properties $P$ and $Q$.

On many occasions, we refer to the ``standard linear-programming reformulation'' of
a problem of the form 
$\min_{H \in \mathbb{R}^{n \times m}}\{\|H\|_1 ~:~ \mathcal{F}(H)=F\}$, where $\mathcal{F}$ is a linear operator and $F$ is constant (matching the output shape of $\mathcal{F}$).
By this we mean 
\[
\textstyle\min_{H \in \mathbb{R}^{n \times m}}\{\langle J,H^+\rangle +  \langle J,H^-\rangle ~:~ \mathcal{F}(H^+)-\mathcal{F}(H^-)=F,\, H^+\geq 0,\, H^-\geq 0 \},
\]
where a solution $(\bar{H}^+,\bar{H}^-)$
yields a 1-norm minimizing $\bar{H}:=\bar{H}^+-\bar{H}^-$
(i.e., a solution of $\min_{H \in \mathbb{R}^{n \times m}}\{\|H\|_1 ~:~ \mathcal{F}(H)=F\}$).
Notice how this linear-programming formulation is 
in ``standard form'' (equations in nonnegative variables).
Similarly, for a problem of the form
$\min_{W,Z}\{\|\mathcal{F}(W,Z)+F\|_1\}$, where $\mathcal{F}$ is a linear operator, $F$ is constant in $\mathbb{R}^{n\times m}$,  
and $\mathcal{F}(W,Z)+F$ models $H$,
the ``standard linear-programming reformulation'' is
\[
\textstyle\min_{H^+ \in \mathbb{R}^{n \times m},W,Z}\{\langle J,H^+\rangle ~:~ 
H^+ -\mathcal{F}(W,Z) \geq F,\, 
H^+ +\mathcal{F}(W,Z) \geq -F
\},
\]
where a solution $(\bar{H}^+,\bar{W},\bar{Z})$
yields a 1-norm minimizing
$\bar{H}:= \mathcal{F}(\bar{W},\bar{Z})+F$
(i.e, a solution of 
$\min_{W,Z} \|\mathcal{F}(W,Z)+F\|_1$).
Notice how the dual of this linear-programming formulation is in standard form, which can be convenient for calculating its solution. 

\medskip

\noindent \textbf{Experimental setup.} Our 
computational experiments are not all in one section, so we describe
our set up here. 

We used the programming language \texttt{Julia} v1.11.6 (using VSCodium v1.97.1), and we ran the experiments on \emph{zebratoo,}
a 32-core machine (running Windows Server 2022 Standard):
two Intel Xeon Gold 6444Y processors running at 3.60GHz, with 16 cores each, and 128 GB of memory. Unless otherwise noted, we solved the optimization problems involved in our experiments with    \texttt{Gurobi} v12.0.1, using parameters:
 Barrier convergence tolerance,
  Feasibility tolerance, and
  Optimality tolerance: $10^{-5}$. The time limit to solve each instance  was 7200 seconds.
To compute the 0-norm, we consider the tolerance for non-zero element as $10^{-5}$.
When presenting the results, we use the symbol `$*$' to indicate that it was not possible to solve the instance within our time~limit.

\section{Sparse universal least-squares solvers
}\label{sec:LS}

A \emph{universal least-squares solver} for a matrix $A\in\mathbb{R}^{m \times n}$ is a
matrix $H\in\mathbb{R}^{n \times m}$ such that 
$\hat{\theta}:=Hb$ is a solution to the least-squares (estimation) problem 
\begin{align}
\textstyle \min_{\theta\in \mathbb{R}^m} \|A\theta -b\|^2_2 \label{lse}\tag{LSE}
\end{align}
for every $b\in\mathbb{R}^m$.

\begin{theorem}[see e.g. 
\protect{\cite[Thm. 1 of Chap. 3]{BenIsrael1974}}]
\label{prop:p1p3}
$H$ satisfies {\rm\ref{P1}+\ref{P3}} (i.e., $H$ is an ah-symmetric generalized inverse of $A$) if and only if $H$ is a universal least-squares solver for $A$.
\end{theorem}

\begin{theorem} [see e.g. 
\protect{\cite[Thm. 3 of Chap. 2]{BenIsrael1974}}]\label{thm:proj_p1p3_Rhode}    
    $H$ satisfies \rm{\ref{P1}+\ref{P3}} if and only if $AH = AA^{\dagger}$, the orthogonal projector onto the range of $A$.
\end{theorem}

The well-known \emph{normal equations}, easily derived, are the least-squares optimality conditions:
$A^\top A\theta = A^\top b$. We are looking for an $H$ so that the least-squares solution is $\hat{\theta}:=Hb$. 
Plugging this into the normal equations
gives $A^\top A Hb = A^\top b$. Now, plugging in all of the standard-unit vectors for $b$, we see that we simply get
\begin{align}
& A^\top A H = A^\top. \tag{PLS}\label{PLS} 
\end{align}
Combining this with Theorem~\ref{prop:p1p3}, we have the following well-known result.
\begin{theorem}\label{lem:P1P3eqPLS}
    $H$ satisfies \rm{\ref{P1}}+\rm{\ref{P3}} if and only if $H$ satisfies \ref{PLS}. 
\end{theorem}

The problem of finding a 1-norm minimizing universal least-squares solver
can be formulated as any of the following optimization problems.
\begin{align}
        & \textstyle\min_{H \in \mathbb{R}^{n \times m}}\{\|H\|_1 : \mbox{\rm{\ref{P1}+\ref{P3}}}\} \tag{$P_{13}^1$}\label{p131} \\
        & \textstyle\min_{H \in \mathbb{R}^{n \times m}}\{\|H\|_1 : AH = AA^{\dagger}\}
        \tag{$P_{\cal{R}}^1$}\label{p1proj13}\\
        & \textstyle\min_{H \in \mathbb{R}^{n \times m}}\{\|H\|_1 : \text{\ref{PLS}}\}
        \tag{$P_{\text{PLS}}^1$}\label{ppls1}\\
        &\textstyle\min_{\jatop{W \in \mathbb{R}^{(n-r) \times (m-r),}}{ Z\in\mathbb{R}^{(n-r) \times r}}}\{\|V_1D^{-1}U_1^\top + V_2WU_2^\top + V_2ZU_1^\top\|_1\} \tag{$\mathcal{P}_{13}^1$} \label{calP13}
\end{align}
Theorem  \ref{prop:p1p3} leads us to the natural formulation \ref{p131}\,.
Theorems
\ref{thm:proj_p1p3_Rhode} and \ref{lem:P1P3eqPLS} lead us to formulations
\ref{p1proj13}\,, \ref{ppls1}\,, respectively.  
The formulation \ref{calP13} is derived from
Theorem \ref{thm:structural}, setting $X:=D^{-1}$ and $Y:=0$ (similar to 
what is done in \cite{PFLX_ORL}).
All four of these formulations can easily be reformulated as LPs, while
previous computational work, using linear programming, was only for \ref{p131}\,.
It is easy to see that the first three formulations have identical feasible regions, but linear-programming solvers could have very different behaviors on the algebraic formulations. 
The following sparsity bound, due to \cite{XFLPsiam},
which applies to the extreme points of the first three of these formulations (the ones with $H$ as the variable) 
is derived from analyzing
 \ref{p1proj13} (as can be seen in the first line of its proof in \cite{XFLPsiam}).
 
\begin{theorem}[\protect{\cite[Prop. 3.1.1]{XFLPsiam}}]\label{prop:basicguarantee}
Suppose that $A\in \mathbb{R}^{m\times n}$ has rank $r$.
 Extreme solutions of the standard linear-programming reformulation of
\ref{p131}
 have at most $mr$ nonzeros. Furthermore, the bound  is sharp for all  $m\geq n\geq r\geq 1$.  
\end{theorem}

\subsection{DRS for \texorpdfstring{\ref{p131}}{P1PLS}}\label{sec:DRS13}

We are interested in solving a problem in the form
\begin{equation}
    \textstyle\min_{H \in \mathbb{R}^{n \times m}}\{ f(H) + g(H)\}, \label{prob:unconstraintedH}
\end{equation}
where $f, g : \mathbb{R}^{n \times m} \rightarrow \mathbb{R} \cup \{ \infty \}$ are convex, closed and proper. To solve \eqref{prob:unconstraintedH}, 
we will apply the Douglas-Rachford Splitting (DRS) algorithm (see, for example,  \cite{dossal2024optimizationorderalgorithms}, \cite{BauschkeCombettes2017}), by iteratively computing, for $k=0,1,\ldots$, 
 \[
    \begin{array}{ll}
            &H^{k+1/2} := \prox_{\lambda f}(V^k), \\
            &V^{k+1/2} := 2H^{k+1/2} - V^k, \\
            &H^{k+1} := \prox_{\lambda g}(V^{k+1/2}), \\
            &V^{k+1} := V^{k} + H^{k+1} - H^{k+1/2},
    \end{array}
\]
where  for $h: \mathbb{R}^{n \times m} \rightarrow \mathbb{R} \cup \{ \infty \}$, convex, closed, and proper, 
\begin{equation*}
    \textstyle\prox_{\lambda h}(V) := \argmin_H\{h(H) + \frac{1}{2\lambda}\|H - V\|_{F}^{2}\}.
\end{equation*}
Initialization amounts to choosing $V^0$, which we discuss in \S\ref{sec:initialpoint}.

In our situation, $f := \|\cdot\|_1$ and $g := \mathcal{I}_{C}$ (the characteristic function of an affine set $\mathcal{C}$).
  We note that $f, g : \mathbb{R}^{n \times m} \rightarrow \mathbb{R} \cup \{ \infty \}$ are convex, closed and proper. 
 From \cite{Parikh_Boyd}, we have  that the proximal operators for $f$ and $g$ are given by
\begin{align*}
    &\prox_{\lambda f}(V) = S_{\lambda}(V),\\
    &\prox_{\lambda g}(V) = \Pi_{\mathcal{C}}(V),
\end{align*}
where $S_{\lambda}(\cdot)$ is the element-wise soft thresholding operator, defined  as
\begin{equation}\label{softthresholding}
    S_\kappa(a) := \left\{\begin{array}{ll}
    a - \kappa, \quad & a > \kappa;\\ 
    0, & |a| \leq \kappa;\\ 
    a + \kappa, &a < -\kappa,
\end{array}\right.
\end{equation}
and $\Pi_{\mathcal{C}}(\cdot)$ is the projection onto the affine set $\mathcal{C}$.

We note that the DRS algorithm can be seen as seeking to generate a sequence $V^0$, $V^1$, $V^2$, $\ldots$ , that converges to a fixed-point of the mapping defined by $F(V):= V + \Pi_{\mathcal{C}}( 2  S_{\lambda}(V) - V) -  S_{\lambda}(V)$.

Finally, we should address some details before implementing DRS to solve \eqref{prob:unconstraintedH}. The first is the projection $\Pi_{\mathcal{C}}(\cdot)$, which in our case, where $\mathcal{C}$ is affine, admits a closed-form solution. Furthermore, we should address the initialization of the algorithm as well as its stopping criteria, for which we may need to evaluate the residuals. The primal residual can be easily calculated using the equations that formulate $\mathcal{C}$, while the dual residual requires us to solve a specific problem that in many cases, such as when $\mathcal{C}$ is affine, admits a closed-form solution.  We discuss these points in the remainder of this section.

\medskip

\subsubsection{Projection onto the affine set $\mathcal{C}$ and residuals}\label{sec:proj_p13}

Here, we consider applying DRS specifically to the reformulation of \ref{p131} given by \ref{ppls1}\,. In this case, we have  $\mathcal{C} := \{H : A^\top AH = A^\top\}$. In the following, we address how to compute a projection onto this affine set and how to compute the primal and dual residuals, based on \cite{Fu_2020}. 

\begin{proposition}\label{projP13}
    If $\mathcal{C} := \{H : A^\top AH = A^\top\}$, for $A\in\mathbb{R}^{m\times n}$, then 
    \[
    \Pi_{\mathcal{C}}(V) = V - A^\dagger A V   + A^\dagger.
    \]
\end{proposition}

\begin{proof}
    Let $V \in \mathbb{R}^{n \times m}$ be a given matrix. We want to find the closed-form solution of
    \begin{equation}\label{prob:proj_p13}
      \textstyle\min_{H \in \mathbb{R}^{n \times m}}\{\|H-V\|_F^2 : A^\top AH = A^\top\}.\tag{$\text{Proj}_{{13}}$}
    \end{equation}
   Using Lagrange multiplier $\Lambda\in\mathbb{R}^{n\times m}$,
   the Lagrangian of \ref{prob:proj_p13} is given by
    \begin{equation*}
        \mathcal{L}(H, \Lambda) := \|H-V\|_F^2 + \langle \Lambda, A^\top AH - A^\top \rangle.
    \end{equation*}
    As \ref{prob:proj_p13} is convex,  its pair of primal and dual solutions $(H, \Lambda)$ satisfies $ \nabla_H \mathcal{L}(H, \Lambda)=0$, or equivalently 
    \begin{equation}\label{grad13zero}
            H  = V - \textstyle\frac{1}{2}(A^\top A\Lambda ).
    \end{equation}
    From \eqref{grad13zero} and $A^\top AH = A^\top$, we have 
    \begin{equation*}
        \begin{aligned}
            A^\top & = A^\top AV - \textstyle\frac{1}{2}A^\top A(A^\top A\Lambda) & \Leftrightarrow \\
            A^\top A(A^\top A\Lambda) & = 2(A^\top AV - A^\top) & \Rightarrow \\
            {A^\dagger}^\top A^\top AA^\top A\Lambda & = 2{A^\dagger}^\top(A^\top AV - A^\top) & \Leftrightarrow \\
            AA^\top A\Lambda & = 2(AV - AA^\dagger) & \Rightarrow \\
            A^\dagger AA^\top A\Lambda & = 2(A^\dagger AV - A^\dagger) & \Leftrightarrow \\
            A^\top A\Lambda & = 2(A^\dagger AV - A^\dagger) & \Rightarrow \\
            {A^\dagger}^\top A^\top A\Lambda & = 2{A^\dagger}^\top(A^\dagger AV - A^\dagger) & \Leftrightarrow \\
            A\Lambda & = 2({A^\dagger}^\top V - {A^\dagger}^\top A^\dagger) & \Rightarrow \\
            A^\dagger A\Lambda & = 2A^\dagger({A^\dagger}^\top V - {A^\dagger}^\top A^\dagger) & \Leftrightarrow \\
            I_m \otimes (A^\dagger A)\vvec(\Lambda) & = 2\vvec(A^\dagger({A^\dagger}^\top V - {A^\dagger}^\top A^\dagger)) .\\
        \end{aligned}
    \end{equation*}
    As $I_m \otimes (A^\dagger A)$ is an orthogonal projection,  $\hat \Lambda = 2A^\dagger({A^\dagger}^\top V - {A^\dagger}^\top A^\dagger)$ solves the last equation above. Substituting it into \eqref{grad13zero} we obtain 
    \begin{equation*}
        \begin{aligned}
            H & = V - A^\top A(A^\dagger{A^\dagger}^\top V - A^\dagger{A^\dagger}^\top A^\dagger)  = V - A^\dagger A V + A^\dagger. 
        \end{aligned}
    \end{equation*}
    Furthermore, we can easily verify that the $H$ obtained is feasible to \ref{prob:proj_p13}\,. 
\qed
\end{proof}

Following \cite{Fu_2020}, we define the primal residual  at iteration $k$ of the DRS algorithm as
\begin{equation*}
    r_p^{k} :=  A^\top AH^{k+1/2} - A^\top.
\end{equation*}
To derive the dual residual at iteration $k$, we first consider the Lagrangian of  \ref{ppls1} and its subdifferential with respect to $H$:
\begin{equation*}
    \mathcal{L}(H, \Lambda) := \|H\|_1 + \langle \Lambda, A^\top AH - A^\top \rangle,
\end{equation*}
\begin{equation*}
    \partial_H \mathcal{L}(H, \Lambda) = \partial \|H\|_1 + A^\top A\Lambda.
\end{equation*}
The optimality conditions for \ref{ppls1} require that 
\begin{equation*}
    0 \in \partial \|H\|_1 + A^\top A\Lambda.
\end{equation*}
From $\textstyle\prox_{\lambda \|\cdot\|_1}(V) =  \textstyle\min_{H \in \mathbb{R}^{n \times m}}\{\|H\|_1
+\frac{1}{2\lambda}\|H - V\|_F^2\}$,  we have that 
\begin{equation*}
       \textstyle H  = \prox_{\lambda \|\cdot\|_1}(V) ~ \Leftrightarrow ~
        0  \in \partial\|H\|_1 + \frac{1}{\lambda}(H - V)  ~ \Leftrightarrow ~
        \frac{1}{\lambda}(V-H)  \in \partial \|H\|_1\,.
\end{equation*}
Then, considering that at iteration $k$ we have $H^{k+1/2} := \prox_{\lambda f}(V^k)$, our aim is to find $\Lambda$ such that
\begin{equation*}
    \textstyle\frac{1}{\lambda}(V^k-H^{k+1/2}) + A^\top A\Lambda = 0,
\end{equation*}
or at least gives us the smallest dual residual, which is equivalent to choosing $\Lambda$ by solving the following least-squares problem:
\begin{equation}\label{prob_13:dual_variable_ls2}
    \textstyle\argmin_{\Lambda\in\mathbb{R}^{n\times m}}\|A^\top A\Lambda  - \frac{1}{\lambda}(H^{k+1/2}-V^k)\|_F^2\,.
\end{equation}
The solution of \eqref{prob_13:dual_variable_ls2} is $\hat\Lambda := \textstyle\frac{1}{\lambda}A^\dagger{A^\dagger}^\top (H^{k+1/2}-V^k)$. 
Finally, evaluating the objective in  \eqref{prob_13:dual_variable_ls2} at $\hat\Lambda$, we obtain the following expression for the dual residual for \ref{ppls1}\,, at iteration $k$ of the DRS algorithm:
\begin{equation*}
    \begin{aligned}
        r_d^{k}  &:=  \textstyle \frac{1}{\lambda}(V^k-H^{k+1/2}) + \frac{1}{\lambda}A^\top AA^\dagger{A^\dagger}^\top (H^{k+1/2}-V^k)  \\
        &~=  \textstyle\frac{1}{\lambda}(V^k-H^{k+1/2}) + \frac{1}{\lambda}A^\top {A^\dagger}^\top (H^{k+1/2}-V^k)\\
        &~=  \textstyle\frac{1}{\lambda}(I- {A^\dagger}A) (V^k-H^{k+1/2}).\\
    \end{aligned}
\end{equation*}

\subsubsection{Initialization}\label{sec:initialpoint}

We define $V^0:=A^\dagger$, which is a feasible solution for \ref{ppls1}\,, requires no additional computational cost as it is already used in the calculation of the projection $\Pi_{\mathcal{C}}(\cdot)$ and the dual residual, and according to our tests, it presented the best performance. We also performed some preliminary experiments with other initializations, but the algorithm converged faster with $V^0:=A^\dagger$.

\subsubsection{Stopping criteria}\label{sec:stop_crit}

We consider two versions of our DRS algorithm with distinct stopping criteria, where we assume that $\epsilon^{\rm abs}, \epsilon^{\rm rel}$ are positive given tolerances.

In the first one, which we refer to as DRS$_{\mbox{\tiny res}}$\,,  we adopt the stopping criterion suggested in \cite{Fu_2020}, which considers a pair of absolute and relative tolerances $\epsilon^{\rm abs}, \epsilon^{\rm rel} > 0$, and the pair of primal and dual residuals in the form $r^{k} := (r_p^{k}, r_d^{k})$ computed at iteration $k$. We stop the algorithm if $\|r^{k}\|_F \leq   \epsilon^{\rm abs} + \epsilon^{\rm rel}\|r^0\|_F$\,, $k>0$. This criterion seeks both a small residual 
and a small residual relative  to the  residual of the initial iterate.

In the second one,  
which we refer to as DRS$_{\mbox{\tiny fp}}$\,, 
 we compute $\tau^{k} := H^{k+1} - H^{k+1/2}=V^{k+1} - V^{k}$ at iteration $k$, and we stop the algorithm if  $\|\tau^{k}\|_F \leq  \epsilon^{\rm abs} + \epsilon^{\rm rel}\|V^1-V^0\|_F$\,, $k>0$.
Note that from \cite{BauschkeCombettes2017} we have that if the sequence $V^k$ ($k=0,1,2,\ldots$), converges to a fixed point of $F(V):= V + \Pi_{\mathcal{C}}( 2  S_{\lambda}(V) - V) -  S_{\lambda}(V)$, then both sequences $H^{k+1/2}$ ($k=0,1,2,\ldots$) and $H^{k+1}$ ($k=0,1,2,\ldots$) converge to $H^*$, a solution of problem \eqref{prob:unconstraintedH}.

We observe that, despite the stopping criterion used, we always obtain a primal feasible solution from the DRS algorithms, taking as output the matrix $H^{k+1}$ computed in the last iteration $k$. Note that $H^{k+1}$ is obtained from a projection onto the feasible set of the problem addressed.


\subsection{Numerical experiments}\label{sec:numexp13}

To compare the different methods that we proposed to compute sparse universal least-squares solvers, we used the same set of test instances used in \cite{ponte2024goodfastrowsparseahsymmetric}. The instances have varied sizes and were constructed using the \texttt{MATLAB} function \texttt{sprand} to  randomly generate $m\!\times\! n$  matrices $A$ with  rank $r$.
They are divided into two categories related to $m$ with $n: = 0.5m$, $r := 0.25m$; small instances with $m := 100,\, 200,\dots,\,500$ 
and large instances with $m := 1000,\, 2000,\dots,\, 5000$.  

In our experiments, we used the solver \texttt{Gurobi} and our implementations of DRS, namely DRS$_{\mbox{\tiny res}}$ and DRS$_{\mbox{\tiny fp}}$\,, with the two stopping criteria proposed in \S\ref{sec:stop_crit}. The projection $\Pi_{\mathcal{C}}(\cdot)$, and the primal and dual residuals, are specified in \S\ref{sec:proj_p13}. In all the experiments with DRS algorithms, we set the parameter $\lambda := 10^{-2}$.

All results presented for \texttt{Gurobi} were obtained with the solution of \ref{ppls1}\,; with this formulation, \texttt{Gurobi} could solve the largest set of instances in our time limit, when we performed a preliminary experiment. The results for this experiment can be found in the Appendix.  It is interesting that the runtime varies significantly for the different linear-programming formulations developed for the problem. 
Despite the better performance of \ref{ppls1}\,, we note that \texttt{Gurobi} does not scale well compared to our implementations of the DRS algorithms, solving only the four smallest instances  within our time limit. 

After observing that DRS$_{\mbox{\tiny fp}}$ was the best performing method, converging faster than DRS$_{\mbox{\tiny res}}$ for solutions with similar 1-norms, we decided to report the results of an experiment where 
we first solved all instances with DRS$_{\mbox{\tiny fp}}$ and saved $\|H_{\mbox{\tiny fp}}\|_1$\,, the 1-norm of the solutions obtained. Then, we ran DRS$_{\mbox{\tiny res}}$\,, stopping the algorithm when  
the obtained solution $H$ satisfied $(\|H\|_1 - \|H_{\mbox{\tiny fp}}\|_1)/ \|H_{\mbox{\tiny fp}}\|_1 \le 10^{-5}$.

In Table \ref{tab:stats_p13}, we compare the results for \texttt{Gurobi} and DRS$_{\mbox{\tiny res}}$ with the results for DRS$_{\mbox{\tiny fp}}$\,.  For DRS$_{\mbox{\tiny fp}}$\,, we show the 0-norm and 1-norm of the solution of each instance. For \texttt{Gurobi} and DRS$_{\mbox{\tiny res}}$\,, we consider the solution $H$ obtained by each algorithm and compare it with the solution $H_{\mbox{\tiny fp}}$ obtained by DRS$_{\mbox{\tiny fp}}$\,, showing the factors $(\|H\| - \|H_{\mbox{\tiny fp}}\|)/\|H_{\mbox{\tiny fp}}\|$ for the 0-norm. For \texttt{Gurobi} we show the factor for the 1-norm as well (note that the 1-norm  factor is always about $10^{-5}$ for DRS$_{\mbox{\tiny res}}$ because of its stopping criterion). We also show the runtimes (in seconds) of \texttt{Gurobi} and the DRS algorithms. 

From the results in Table \ref{tab:stats_p13}, we see that although \texttt{Gurobi} finds a solution with a better 0-norm and a slightly better 1-norm, its runtime makes it uncompetitive in solving our problem. We also see that computing the residuals to adopt the stopping criterion of DRS$_{\mbox{\tiny res}}$ in our implementation causes its runtime to be much larger than the runtime  of DRS$_{\mbox{\tiny fp}}$ on average.
However, despite its worse performance, the stopping criterion of DRS$_{\mbox{\tiny res}}$ was satisfied for all instances within the time limit. Finally, we observe that DRS$_{\mbox{\tiny res}}$ obtains solutions with worse 0-norms, except for the smallest instance.

\begin{table}[ht!]
    \centering  
    \caption{Comparison against DRS$_{\mbox{\tiny fp}}$ for \ref{p131} ($n=0.5m$, $r=0.25m$)}
    \label{tab:stats_p13}
\begin{footnotesize}
        \begin{tabular}{r|rrr|rrr|rr}        
         \multicolumn{1}{c}{} & \multicolumn{3}{c|}{\texttt{Gurobi}} & \multicolumn{3}{c|}{DRS$_{\mbox{\tiny fp}}$} & \multicolumn{2}{c}{DRS$_{\mbox{\tiny res}}$} \\
        \hline       
        \multicolumn{1}{c|}{$m$} & \multicolumn{1}{c}{$\vphantom{\mbox{$X^{X^X}$}} \jatop{\vphantom{\mbox{$X^{X^X}$}}\|H\|_0}{\mbox{factor}}$} & \multicolumn{1}{c}{$\jatop{\|H\|_1}{\mbox{factor}}$} & \multicolumn{1}{c|}{time} & \multicolumn{1}{c}{$\|H\|_0$} & \multicolumn{1}{c}{$\|H\|_1$} & \multicolumn{1}{c|}{time} & \multicolumn{1}{c}{$\jatop{\|H\|_0}{\mbox{factor}}$} & \multicolumn{1}{c}{time} \\[4pt]
         \hline    \vphantom{\mbox{$X^{X^X}$}} 
        100 & -5.77e-2 & -7.61e-5 & 2.25 & 2565 & 182.35 & 0.23 & -1.09e-2 & 0.23 \\
        200 & -7.37e-2 & -7.70e-5 & 10.42 & 10554 & 504.89 & 1.69 & 1.51e-2 & 2.66 \\
        300 & -7.07e-2 & -9.61e-5 & 43.45 & 23672 & 816.25 & 2.64 & 6.38e-3 & 4.66 \\
        400 & -9.10e-2 & -1.37e-4 & 235.34 & 42762 & 1246.57 & 3.20 & 1.19e-2 & 10.83 \\
        500 & \multicolumn{1}{c}{$*$} & \multicolumn{1}{c}{$*$} & \multicolumn{1}{c|}{$*$} & 64977 & 1736.15 & 4.95 & 7.28e-3 & 20.66 \\
        \hline
        \vphantom{\mbox{$X^{X^X}$}} 1000 & \multicolumn{1}{c}{$*$} & \multicolumn{1}{c}{$*$} & \multicolumn{1}{c|}{$*$} & 263597 & 4541.65 & 26.79 & 6.92e-3 & 58.97 \\
        2000 & \multicolumn{1}{c}{$*$} & \multicolumn{1}{c}{$*$} & \multicolumn{1}{c|}{$*$} & 1013012 & 11002.67 & 91.68 & 3.35e-3 & 228.66 \\
        3000 & \multicolumn{1}{c}{$*$} & \multicolumn{1}{c}{$*$} & \multicolumn{1}{c|}{$*$} & 2224556 & 19076.73 & 174.10 & 2.31e-3 & 473.47 \\
        4000 & \multicolumn{1}{c}{$*$} & \multicolumn{1}{c}{$*$} & \multicolumn{1}{c|}{$*$} & 3853967 & 28058.49 & 362.91 & 8.25e-4 & 803.34 \\
        5000 & \multicolumn{1}{c}{$*$} & \multicolumn{1}{c}{$*$} & \multicolumn{1}{c|}{$*$} & 6031288 & 38046.54 & 483.86 & 1.21e-3 & 1053.28 \\
    \end{tabular}
    \end{footnotesize}
\end{table}

We note that while 1-norm minimization is commonly used to induce sparsity, there is no guarantee that this actually leads to sparser solutions. Therefore, it is interesting to verify the effectiveness of the induction method in our experiments. We observe that the smaller the 1-norms of our solutions, the sparser they are.

In Table \ref{tab:stats2_p13}, we compare the solution $H_{\mbox{\tiny fp}}$ obtained by DRS$_{\mbox{\tiny fp}}$ with the Moore-Penrose pseudoinverse of $A$, showing the factors $(\|H_{\mbox{\tiny fp}}\|- \|A^\dagger\|)/\|A^\dagger\|$ for the 0-norm and the 1-norm. We also show the ratios between the rank of $H_{\mbox{\tiny fp}}$ and $n$, and between the 0-norm of $H_{\mbox{\tiny fp}}$ and the upper bound $\beta_{13}:=mr$ on the number of nonzero elements of extreme solutions of the standard linear-programming reformulation of \ref{p131}
(see Theorem \ref{prop:basicguarantee}). Finally, we show this last ratio for \texttt{Gurobi}'s solutions.

\begin{table}[ht]
    \centering  
    \caption{More statistics for DRS$_{\mbox{\tiny fp}}$ for \ref{p131} ($n=0.5m$, $r=0.25m$)}
    \label{tab:stats2_p13}
\begin{footnotesize}
        \begin{tabular}{r|r|rrrr}      
         \multicolumn{1}{c}{} & \multicolumn{1}{c|}{\texttt{Gurobi}} & \multicolumn{4}{c}{DRS$_{\mbox{\tiny fp}}$} \\ 
        \hline    
        \multicolumn{1}{c|}{$m$} & $\frac{\vphantom{\mbox{$X^{X^X}$}} \|H\|_0}{\beta_{13}}$ & $\frac{\|H\|_0}{\beta_{13}}$ & $\jatop{\|H\|_0}{\mbox{factor}}$ & $\jatop{\|H\|_1}{\mbox{factor}}$ & $\frac{\text{rank}(H)}{n}$ \\ [4pt]
        \hline
 \vphantom{\mbox{$X^{X^X}$}} 
        100 & 0.967 & 1.026 & -0.461 & -0.182 & 0.960 \\
        200 & 0.978 & 1.055 & -0.439 & -0.171 & 0.980 \\
        300 & 0.978 & 1.052 & -0.439 & -0.186 & 0.980 \\
        400 & 0.972 & 1.069 & -0.433 & -0.185 & 0.990 \\
        500 & \multicolumn{1}{c|}{$*$} & 1.040 & -0.446  & -0.182 & 1.000 \\
        \hline
             \vphantom{\mbox{$X^{X^X}$}} 
        1000 & \multicolumn{1}{c|}{$*$} & 1.054 & -0.434 & -0.186 & 0.998 \\
        2000 & \multicolumn{1}{c|}{$*$} & 1.013 & -0.450 & -0.199 & 1.000 \\
        3000 & \multicolumn{1}{c|}{$*$} & 0.989 & -0.456 & -0.201 & 0.999 \\
        4000 & \multicolumn{1}{c|}{$*$} & 0.963 & -0.466 & -0.198 & 1.000 \\
        5000 & \multicolumn{1}{c|}{$*$} & 0.965 & -0.462 & -0.202 & 0.999 \\
    \end{tabular}
    \end{footnotesize}
\end{table}

We can observe from the results in Table \ref{tab:stats2_p13} that by minimizing the 1-norm, we obtain ah-symmetric generalized inverses with the 0-norm reduced by about 55\% and the 1-norm reduced by about 80\%, compared to $A^\dagger$. We see that the number of nonzero elements in the \texttt{Gurobi} solutions is about 97\% of the upper bound $mr$ for extreme solutions. The number of nonzero elements in the DRS$_{\mbox{\tiny fp}}$ solutions is also not far from this bound. Interestingly, we have larger ratios for the smallest instances, but we note that the ratios might decrease if we let DRS$_{\mbox{\tiny fp}}$ run longer, compromising, however, the advantage of its fast convergence.

\section{Sparse universal minimum-rank solvers for least-squares }\label{sec:ah-ref}

We now search for sparse universal least-squares solvers of low rank.  Low rank is an interesting feature for an ah-symmetric generalized inverse $H$ in the context of the least-squares application, as it corresponds to a type of explainability for the associated linear model. The following theorem provides a key result in incorporating this feature into our studies.  

\begin{theorem}
[see e.g. 
\protect{\cite[Lem. 1, part d and Thm. 2 of Chap. 1]{BenIsrael1974}}]
\label{thm:rank} 
If $H$ is a generalized inverse of $A$, then
$\rank(H)\geq \rank(A)$, with equality if and only if  {\rm\ref{P2}} holds (i.e., if and only if $H$ is reflexive).
\end{theorem}

Based on Theorems  \ref{prop:p1p3} and \ref{thm:rank}, we can see that if $H$ is an ah-symmetric reflexive generalized inverse of $A$, then $H$ is a minimum-rank universal least-squares solver for $A$. 
 In  Theorem \ref{thm:propertyPLSr} we characterize the properties \ref{P1}+\ref{P2}+\ref{P3}, of ah-symmetric reflexive generalized inverses, and the linear system \ref{PLSr} defined below.

\begin{theorem}\label{thm:propertyPLSr}
    $H$ satisfies \rm{\ref{P1}+\ref{P2}+\ref{P3}}
    (i.e., $H$ is an ah-symmetric reflexive generalized inverse of $A$) if and only if $H$ satisfies 
    \begin{equation}
    A^\top H^\top A^\top  + HAA^\dagger  = A^\top  + H. \label{PLSr}\tag{PLSr}   
    \end{equation}
\end{theorem}

\begin{proof}
\phantom{.}

\noindent 1. $(\Rightarrow)$: If $H$ satisfies \ref{P1}+\ref{P2}+\ref{P3} then $H$ satisfies \ref{PLSr}.

    Because $H$ satisfies \ref{P1}+\ref{P2}, we have
    \begin{equation*}
            A^\top H^\top A^\top  + HAH  = A^\top  + H  \Rightarrow 
            A^\top H^\top A^\top  + HAA^\dagger   = A^\top  + H,
    \end{equation*}
    where we used that $AH = AA^\dagger $ by Theorem \ref{thm:proj_p1p3_Rhode}, considering that $H$ satisfies \ref{P1}+\ref{P3}.

\noindent     2. $(\Leftarrow)$: If $H$ satisfies \ref{PLSr} then $H$ satisfies \ref{P1}+\ref{P2}+\ref{P3}.

    Consider the full SVD of $A$, given by $A = U\Sigma V^\top $ and the representation of $H$ given by $H = V\Gamma U^\top $ (see Theorem \ref{thm:structural}). Because $H$ satisfies \ref{PLSr}, we have
    \begin{equation*}
        \begin{aligned}
            A^\top H^\top A^\top  + HAA^\dagger  & = A^\top  + H & \Rightarrow \\
            V\Sigma^\top U^\top  U\Gamma^\top V^\top  V\Sigma^\top U^\top  + V\Gamma U^\top  U\Sigma V^\top   V\Sigma^\dagger  U^\top  & = V\Sigma^\top U^\top  + V\Gamma U^\top  & \Rightarrow \\
            V \Sigma^\top \Gamma^\top \Sigma^\top U^\top  + V\Gamma \Sigma \Sigma^\dagger U^\top  & = V(\Sigma^\top  + \Gamma)U^\top  & \Rightarrow \\
            \Sigma^\top \Gamma^\top \Sigma^\top  + \Gamma \Sigma \Sigma^\dagger  & = \Sigma^\top  + \Gamma & \Rightarrow \\
            \begin{bmatrix} D^\top  & 0 \\ 0 & 0 \end{bmatrix}\begin{bmatrix} X^\top  & Z^\top  \\ Y^\top  & W^\top  \end{bmatrix}\begin{bmatrix} D^\top  & 0 \\ 0 & 0 \end{bmatrix} \! + \! \begin{bmatrix} X & Y \\ Z & W \end{bmatrix}\begin{bmatrix} D & 0 \\ 0 & 0 \end{bmatrix}\begin{bmatrix} D^{-1} & 0 \\ 0 & 0 \end{bmatrix} & = \begin{bmatrix} D^\top  & 0 \\ 0 & 0 \end{bmatrix} \! + \! \begin{bmatrix} X & Y \\ Z & W \end{bmatrix} & \Rightarrow \\
            \begin{bmatrix} D^\top X^\top D^\top  & 0 \\ 0 & 0 \end{bmatrix} \! + \! \begin{bmatrix} X & 0 \\ Z & 0 \end{bmatrix} & = \begin{bmatrix} D^\top  & 0 \\ 0 & 0 \end{bmatrix} \! + \! \begin{bmatrix} X & Y \\ Z & W \end{bmatrix} & \Rightarrow \\
            \begin{bmatrix} D^\top X^\top D^\top  & 0 \\ 0 & 0 \end{bmatrix} & = \begin{bmatrix} D^\top  & Y \\ 0 & W \end{bmatrix}.
        \end{aligned}
    \end{equation*}
From the last equation, we see that  $Y \!=\! 0$ and $W \!=\! 0$, and $D^\top X^\top D^\top  \!=\! D^\top$, so  $X \!=\! D^{-1}$. Then, by Theorem \ref{thm:structural}, we see that $H$ satisfies \rm{\ref{P1}+\ref{P2}+\ref{P3}}.
   \qed
\end{proof}

The problem of finding a 1-norm minimizing \emph{minimum-rank} universal least-squares solver
can be formulated as any of the following optimization problems:
\begin{align}
   & \textstyle\min_{H \in \mathbb{R}^{n \times m}}\{\|H\|_1 : \rm{\mbox{\ref{P1}+\ref{P2}+\ref{P3}}}\}, \tag{$P_{123}^1$}\label{p1231a} \\
   & \textstyle\min_{H \in \mathbb{R}^{n \times m}}\{\|H\|_1 : \mbox{\rm{\ref{P1}}+$(HAA^\dagger = H)$+\rm{\ref{P3}}}\}\, , 
   \tag{$P_{12_{\mbox{\tiny{lin}}}3}^1$}
   \label{p1231lin} \\
    & \textstyle\min_{H \in \mathbb{R}^{n \times m}}\{\|H\|_1 :\mbox{ $(AH=AA^\dagger)$+$(HAA^\dagger = H)$}\}\, , 
    \tag{$P_{\mathcal{R}2_{\mbox{\tiny{lin}}}}^1$}
    \label{p1proj13p2lin} \\
    & \textstyle\min_{H \in \mathbb{R}^{n \times m}}\{\|H\|_1 : \mbox{\ref{PLS}+$(HAA^\dagger = H)$}\}\, , 
    \tag{$P_{\text{PLS}2_{\mbox{\tiny{lin}}}}^1$}
    \label{p1plsp2lin} \\
    & \textstyle\min_{H \in \mathbb{R}^{n \times m}}\{\|H\|_1 : \text{\ref{PLSr}}\}\, , \tag{$P_{\text{PLSr}}^1$}\label{pplsr1}\\
    & \textstyle\min_{Z\in\mathbb{R}^{(n-r) \times r}}\left\| V_1D^{-1}U_1^\top + V_2ZU_1^\top\right\|_1\,.\tag{\mbox{$\mathcal{P}^1_{123}$}}\label{calP123}
\end{align}

Theorems \ref{prop:p1p3} and \ref{thm:rank} lead us to the natural formulation \ref{p1231a}\,, which we do not actually compute with  because it is nonconvex, due to \ref{P2}.
From \ref{p1231a}\,, linearizing \ref{P2} using Theorem \ref{thm:proj_p1p3_Rhode}  leads 
to \ref{p1231lin}\,. From there, now using Theorem \ref{thm:proj_p1p3_Rhode}
 to 
replace \ref{P1}+\ref{P3} with $AH=AA^\dagger$, we obtain \ref{p1proj13p2lin}.
Going back to  \ref{p1231lin}\,, and instead replacing \ref{P1}+\ref{P3} with 
\ref{PLS} using Theorem \ref{lem:P1P3eqPLS}, we obtain \ref{p1plsp2lin}\,.
The formulation \ref{pplsr1} comes immediately from  \ref{p1231a} via Theorem \ref{thm:propertyPLSr}\,.
 The formulation \ref{calP123}\,,  which can be found in \cite{PFLX_ORL}, is derived from
Theorem \ref{thm:structural}, setting $X:=D^{-1}$, $Y:=0$ and $W:=0$.
All six of these formulations can easily be reformulated as LPs. 
Previous computational work, using linear programming, was only for \ref{p1231a}
and \ref{calP123} (see \cite{PFLX_ORL}). 
Additionally, in \cite{ponte2024goodfastrowsparseahsymmetric}, there is an 
ADMM 
for \ref{calP123}\,.
Finally, the following sparsity bound, due to \cite{XFLPsiam},
which applies to the extreme points of the first five of these formulations (the ones with $H$ as the variable), 
is derived from analyzing
\ref{p1proj13p2lin} (as can be seen in the first line of its proof). 

\begin{theorem}[\protect{\cite[Prop. 3.1.2]{XFLPsiam}}]\label{prop:basicguarantee2}
Suppose that $A\in \mathbb{R}^{m\times n}$ has rank $r$.
 Extreme solutions of the standard linear-programming reformulation of
\ref{p1231a}  have at most $mr+(m-r)(n-r)$ nonzeros.
\end{theorem}

\subsection{ADMM for \texorpdfstring{\ref{p1231a}}{P123}}\label{sec:admm123}

We use an ADMM 
for \ref{p1231a}\,, applying it to the reformulation \ref{calP123} from \cite{ponte2024goodfastrowsparseahsymmetric}. Initially, by introducing a variable $H \in \mathbb{R}^{n \times m}$, we rewrite  \ref{calP123}  as
\begin{equation}\label{prob:admm_mat_1norm}
 \textstyle\min_{H \in \mathbb{R}^{n \times m}}\left\{
 \|H\|_{1} ~:~  H =  V_1D^{-1}U_1^\top + V_2ZU_1^\top
 \right\}.
\end{equation}
The augmented Lagrangian associated with \eqref{prob:admm_mat_1norm} is
\begin{align*}
\mathcal{L}_\rho(Z,H,\Lambda)\!
 &:=\! 
 \|H\|_{1} \!+\! \textstyle \frac{\rho}{2}\!\left\|V_1D^{-1}U_1^\top \!+\! V_2ZU_1^\top \!-\!H \!+\!\Lambda \right\|^2_F - 
  \frac{\rho}{2}\left\|\Lambda\right\|_F^2\,,
\end{align*}
where $\rho >0$ is the penalty parameter
and  $\Lambda $ is the scaled Lagrange multiplier; that is,  $\Theta := \rho\Lambda$ is the Lagrange multiplier associated to the constraint in \eqref{prob:admm_mat_1norm}. We apply ADMM
to \ref{calP123} by iteratively solving, for $k=0,1,\ldots,$ 
\begin{align}
    &Z^{k+1}:=\textstyle\argmin_Z ~ \mathcal{L}_\rho(Z,H^{k},\Lambda^k),\label{eq:Zmin1normsubpa}\\
    &H^{k+1}:=\textstyle\argmin_H ~ \mathcal{L}_\rho(Z^{k+1},H,\Lambda^k),\label{eq:Emin1normsubpa}\\
    &\textstyle\Lambda^{k+1}:=\Lambda^{k} + V_1D^{-1}U_1^\top + V_2 Z^{k+1}U_1^\top - H^{k+1}.\nonumber
\end{align}

\noindent {\bf{Update \texorpdfstring{$Z$}{Z}:}} 
Consider  subproblem \eqref{eq:Zmin1normsubpa}, more specifically,
\begin{align}\label{eq:Zmin1normsubprob}
Z^{k+1}:=\textstyle \argmin_{Z}  \left\| J - V_2 Z U_1^\top\right\|^2_F,
\end{align}
where $J:=  H^k -V_1D^{-1}U_1^\top -\Lambda^k$. We can easily  verify that the solution of \eqref{eq:Zmin1normsubprob} is given by $Z^{k+1} =  V_2^\top J U_1$\,.

\medskip

\noindent {\bf{Update \texorpdfstring{$H$}{H}:}} 
Consider  subproblem \eqref{eq:Emin1normsubpa}, more specifically, 
\begin{align}
H^{k+1}:=\textstyle\argmin_H\left\{ \|H\|_{1} + \frac{\rho}{2}\left\|H -Y  \right\|^2_F\right\},\label{eq:Emin1normsubprob}
\end{align}
where $Y := V_1D^{-1}U_1^\top + V_2Z^{k+1}U_1^\top  +\Lambda^k$.

From \cite[Section 4.4.3]
{boyd2011distributed}, we have that the solution of \eqref{eq:Emin1normsubprob} is given by
$H^{k+1}_{ij} = S_{1\!/\!\rho}(Y_{ij})$,
for $i = 1,\dots,n$ and $j = 1,\dots,m$, where  the element-wise soft thresholding operator $S$ is defined in \eqref{softthresholding}.

\medskip

\noindent
{\bf{Initialization:}}   
Set $\Lambda^0:=\hat\Theta/\rho$, where $\textstyle \hat \Theta :=\frac{1}{\|V_1 U_1^\top\|_{\infty}} V_1 U_1^\top$\,, and  $H^0 := V_1D^{-1}U_1^\top + \Lambda^0$.

\medskip

\noindent{\bf{Stopping criterion:}} 
We consider a stopping criterion  from \cite[Section 3.3.1]{boyd2011distributed}, and select an absolute tolerance $\epsilon^{\mbox{\scriptsize abs}}$ and a
relative tolerance $\epsilon^{\mbox{\scriptsize rel}}$. The algorithm stops at iteration $k$ if  
\begin{align*}
    &\|r^{k+1}\|_F \leq \epsilon^{\mbox{\scriptsize abs}}\sqrt{nm}  
    + \epsilon^{\mbox{\scriptsize rel}}\max\left\{\|H^{k+1}\|_F\,, \|V_2Z^{k+1}U_1^\top\|_F\,,\|V_1D^{-1}U_1^\top\|_F\right\}, \\
    &\|s^{k+1}\|_F \leq \epsilon^{\mbox{\scriptsize abs}}
    \sqrt{(n-r)r} +    
\epsilon^{\mbox{\scriptsize rel}}\rho \|V_2^\top\Lambda^{k+1}U_1\|_F\,,
\end{align*}
where $r^{k+1} := V_1D^{-1}U_1^\top + V_2Z^{k+1}U_1^\top - H^{k+1}$ is the primal residual,  and $s^{k+1}:= \rho V_2^\top(H^{k+1} - H^{k})U_1$\, is the dual residual.

\medskip

We note that, despite the tolerance used for the ADMM,
a feasible solution for \ref{p1231a} is constructed, considering the result of Theorem \ref{thm:structural}, and taking as output of ADMM
the matrix $V_1D^{-1}U_1^\top + V_2 Z^{k+1}$, where $Z^{k+1}$ was obtained in its last iteration.


\subsection{DRS for \texorpdfstring{\ref{p1231a}}{P123}}\label{sec:proj_p123}

Here, we consider applying DRS to \ref{p1231a}\,, more specifically, to its reformulation \ref{p1plsp2lin}\,. In this case, we have  $\mathcal{C} := \{H : A^\top AH = A^\top, HAA^\dagger = H\}$. The only differences between this application of DRS and the one detailed in \S\ref{sec:DRS13} are how to compute a projection onto the new affine set $\mathcal{C}$ and how to compute the primal and dual residuals, which we discuss in the following. 

\begin{proposition}\label{projP123}
    If $\mathcal{C} := \{H : A^\top AH = A^\top, HAA^\dagger = H\}$, for $A\in\mathbb{R}^{m\times n}$, then 
    \[
    \Pi_{\mathcal{C}}(V) = A^\dagger - A^\dagger AVAA^\dagger + VAA^\dagger.
    \]
\end{proposition}

\begin{proof}
    Let $V \!\in\! \mathbb{R}^{n \times m}$ be a given matrix. We want to find the closed-form solution~of
    \begin{equation}\label{prob:proj_p123}
        \textstyle\argmin_H\{\|V-H\|_F^2 : A^\top AH = A^\top, HAA^\dagger = H\}. \tag{$\text{Proj}_{123}$}
    \end{equation}
  Using Lagrange multipliers $\Lambda,\Psi\in\mathbb{R}^{n\times m}$,
   the Lagrangian of 
 \ref{prob:proj_p123} is given by
    \begin{equation*}
        \mathcal{L}(H, \Lambda, \Psi) := \|V-H\|_F^2 + \langle \Lambda, A^\top AH - A^\top \rangle + \langle \Psi, HAA^\dagger - H \rangle.
    \end{equation*}
    As \ref{prob:proj_p123} is convex, we have that its primal-dual solution $(H, \Lambda, \Psi)$ satisfies $\nabla_H \mathcal{L}(H, \Lambda, \Psi)=0$, or equivalently 
    \begin{equation}\label{grad123zero}
     H  = V - \textstyle\frac{1}{2}(A^\top A\Lambda + \Psi AA^\dagger - \Psi).
     \end{equation}
     From \eqref{grad123zero} and  $A^\top AH = A^\top$ we have 
    \begin{equation*}
        \begin{aligned}
            A^\top & = A^\top AV - \textstyle\frac{1}{2}A^\top A(A^\top A\Lambda + \Psi AA^\dagger - \Psi) & \Rightarrow \\
            A^\top A & = A^\top AVA - \textstyle\frac{1}{2}A^\top AA^\top A\Lambda A & \Leftrightarrow \\
            A^\top AA^\top A\Lambda A & = 2(A^\top AVA - A^\top A) & \Rightarrow \\
            {A^\dagger}^\top A^\top AA^\top A\Lambda A & = 2{A^\dagger}^\top(A^\top AVA - A^\top A) & \Leftrightarrow \\
            AA^\top A\Lambda A & = 2(AVA - A) & \Rightarrow \\
            A^\dagger AA^\top A\Lambda A & = 2(A^\dagger AVA - A^\dagger A) & \Leftrightarrow \\
            A^\top A\Lambda A & = 2(A^\dagger AVA - A^\dagger A) & \Rightarrow \\
            {A^\dagger}^\top A^\top A\Lambda A & = 2{A^\dagger}^\top(A^\dagger AVA - A^\dagger A) & \Leftrightarrow \\
            A\Lambda A & = 2({A^\dagger}^\top VA - {A^\dagger}^\top) & \Rightarrow \\
            A^\dagger A\Lambda AA^\dagger & = 2A^\dagger({A^\dagger}^\top VA - {A^\dagger}^\top)A^\dagger & \Leftrightarrow \\
            (AA^\dagger) \otimes (A^\dagger A)\vvec(\Lambda) & = 2\vvec(A^\dagger({A^\dagger}^\top VA - {A^\dagger}^\top)A^\dagger). 
        \end{aligned}
    \end{equation*}
As $(AA^\dagger) \otimes (A^\dagger A)$ is an orthogonal projection,  $\hat \Lambda := 2A^\dagger({A^\dagger}^\top VA - {A^\dagger}^\top)A^\dagger$ solves the last equation above. 
We note that $\hat \Lambda = \hat \Lambda AA^\dagger$. 

Similarly, 
 from \eqref{grad123zero} and  $HAA^\dagger = H$, and now considering $\Lambda:=\hat\Lambda$, we have 
    \begin{equation*}
        \begin{aligned}
            H & = VAA^\dagger -\textstyle\frac{1}{2}A^\top A\Lambda & \Leftrightarrow \\
            V -\textstyle\frac{1}{2}(A^\top A\Lambda + \Psi AA^\dagger - \Psi) & = VAA^\dagger -\textstyle\frac{1}{2}A^\top A\Lambda & \Leftrightarrow \\
            VAA^\dagger - V & = -\textstyle\frac{1}{2}(\Psi AA^\dagger - \Psi) & \Leftrightarrow \\
            \Psi AA^\dagger - \Psi & = 2(V - VAA^\dagger) & \Leftrightarrow \\
            ((AA^\dagger) \otimes I_n - I_m \otimes I_n)\vvec(\Psi) & = 2\vvec(V - VAA^\dagger) & \Leftrightarrow \\
            (I_m \otimes I_n - (AA^\dagger) \otimes I_n)\vvec(\Psi) & = 2\vvec(VAA^\dagger - V).
        \end{aligned}
    \end{equation*}
    As $I_m \otimes I_n - (AA^\dagger) \otimes I_n$ is an orthogonal projection, $\hat\Psi = 2(VAA^\dagger - V)$ solves the last equation above.

Finally, substituting $\hat\Lambda$ and $\hat\Psi$ into \eqref{grad123zero}, we obtain 
    \begin{equation*}
        \begin{aligned}
            H& = V - A^\top A(A^\dagger{A^\dagger}^\top VAA^\dagger - A^\dagger{A^\dagger}^\top A^\dagger) - (VAA^\dagger - V)AA^\dagger + (VAA^\dagger - V) \\
            & = V - A^\top {A^\dagger}^\top VAA^\dagger + A^\top{A^\dagger}^\top A^\dagger + VAA^\dagger - V \\
            & = A^\dagger - A^\dagger AVAA^\dagger + VAA^\dagger.
        \end{aligned}
    \end{equation*}
     Furthermore, we can easily verify that the $H$ obtained is feasible to \ref{prob:proj_p123}\,. 
    \qed
\end{proof}

From Theorems \ref{lem:P1P3eqPLS}, \ref{thm:proj_p1p3_Rhode} and \ref{thm:propertyPLSr}, we see that the constraints in \ref{p1plsp2lin} are equivalent to  \ref{PLSr}\,.
So, following \cite{Fu_2020}, we define the primal residual  at iteration $k$  of the DRS algorithm as
\begin{equation*}\label{rp:p123}
    r_p^{k} := A^\top {H^{k+1/2}}^\top A^\top + H^{k+1/2}AA^\dagger - A^\top - H^{k+1/2}.
\end{equation*}

To derive the dual residual at iteration $k$, we first consider the Lagrangian of  \ref{p1plsp2lin} and its subdifferential with respect to $H$:
\begin{equation*}
    \mathcal{L}(H, \Lambda, \Psi) := \|H\|_1 + \langle \Lambda, A^\top AH - A^\top \rangle + \langle \Psi, HAA^\dagger - H \rangle,
\end{equation*}
\begin{equation*}\label{prob_123:subdifflagrangian}
    \partial_H \mathcal{L}(H, \Lambda, \Psi) = \partial \|H\|_1 + A^\top A\Lambda + \Psi AA^\dagger - \Psi.
\end{equation*}

Then, following the same steps of \S\ref{sec:DRS13}, 
we choose $\Lambda, \Psi$ that solve the least-squares problem
\begin{equation}\label{prob_123:dual_variable_ls}
    \textstyle\argmin_{\Lambda, \Psi}\|A^\top A\Lambda + \Psi AA^\dagger - \Psi - \frac{1}{\lambda}(H^{k+1/2}-V^k)\|_F^2\,.
\end{equation}

\begin{proposition}
    The solution of \eqref{prob_123:dual_variable_ls} is $\hat\Lambda := \textstyle\frac{1}{\lambda}A^\dagger{A^\dagger}^\top (H^{k+1/2}-V^k)AA^\dagger$ and $\hat\Psi := \frac{1}{\lambda}(H^{k+1/2}-V^k)AA^\dagger - \frac{1}{\lambda}(H^{k+1/2}-V^k)$.
\end{proposition}

\begin{proof}
    Let $L(\Lambda, \Psi) := \|A^\top A\Lambda + \Psi AA^\dagger - \Psi - B\|_F^2$\,, where $B := \frac{1}{\lambda}(H^{k+1/2}-V^k)$. As $L$  is a convex function, its minimum is attained at $(\hat\Lambda, \hat\Psi)$ that satisfies $\nabla_\Lambda L(\hat\Lambda, \hat\Psi) = 0$ and $\nabla_\Psi L(\hat\Lambda, \hat\Psi) = 0$. Computing first the gradient with respect to $\Lambda$, we have
    \begin{equation*} 
        \nabla_\Lambda L(\Lambda, \Psi) = 2A^\top A(A^\top A\Lambda + \Psi AA^\dagger - \Psi - B).
    \end{equation*}

    From $\nabla_\Lambda L(\Lambda, \Psi) = 0$, we have

    \begin{equation*}
        \begin{aligned}
            2A^\top A(A^\top A\Lambda + \Psi AA^\dagger - \Psi - B) & = 0 & \Rightarrow \\
            2A^\top A(A^\top A\Lambda A + \Psi AA^\dagger A - \Psi A - BA) & = 0 & \Leftrightarrow \\
            2A^\top AA^\top A\Lambda A & = 2A^\top ABA & \Rightarrow \\
            {A^\dagger}^\top A^\top AA^\top A\Lambda A & = {A^\dagger}^\top A^\top ABA & \Leftrightarrow \\
            AA^\top A\Lambda A & = ABA & \Rightarrow \\
            A^\dagger AA^\top A\Lambda A & = A^\dagger ABA & \Leftrightarrow \\
            A^\top A\Lambda A & = A^\dagger ABA & \Rightarrow \\
            {A^\dagger}^\top A^\top A\Lambda A & = {A^\dagger}^\top A^\dagger ABA & \Leftrightarrow \\
            A\Lambda A & = {A^\dagger}^\top BA & \Rightarrow \\
            A^\dagger A\Lambda AA^\dagger & = A^\dagger{A^\dagger}^\top BAA^\dagger & \Leftrightarrow \\
            (AA^\dagger) \otimes (A^\dagger A) \vvec(\Lambda) & = \vvec(A^\dagger{A^\dagger}^\top BAA^\dagger).
        \end{aligned}
    \end{equation*}

    As $(AA^\dagger) \otimes (A^\dagger A)$ is an orthogonal projection,  $\Lambda = A^\dagger{A^\dagger}^\top BAA^\dagger$ solves this equation.

    Computing the gradient of $L$ with respect to $\Psi$, we have
\begin{equation*} 
        \nabla_\Psi L(\Lambda, \Psi) = 2(A^\top A\Lambda + \Psi AA^\dagger - \Psi - B)AA^\dagger - 2(A^\top A\Lambda + \Psi AA^\dagger - \Psi - B).
    \end{equation*}

    From $\nabla_\Psi L(\Lambda, \Psi) = 0$, we have

    \begin{equation*}
        \begin{aligned}
            (A^\top A\Lambda + \Psi AA^\dagger - \Psi - B)AA^\dagger & = A^\top A\Lambda + \Psi AA^\dagger - \Psi - B & \Leftrightarrow \\
            A^\top A\Lambda AA^\dagger - BAA^\dagger & = A^\top A\Lambda + \Psi AA^\dagger - \Psi - B & \Leftrightarrow \\
            A^\top AA^\dagger{A^\dagger}^\top BAA^\dagger AA^\dagger - BAA^\dagger & = A^\top AA^\dagger{A^\dagger}^\top BAA^\dagger + \Psi AA^\dagger - \Psi - B & \Leftrightarrow \\
            A^\top {A^\dagger}^\top BAA^\dagger - BAA^\dagger & = A^\top {A^\dagger}^\top  BAA^\dagger + \Psi AA^\dagger - \Psi - B & \Leftrightarrow \\
            \Psi AA^\dagger - \Psi & = B - BAA^\dagger & \Leftrightarrow \\
            ((AA^\dagger) \otimes I_n - I_m \otimes I_n)\vvec(\Psi) & = \vvec(B - BAA^\dagger) & \Leftrightarrow \\
            (I_m \otimes I_n - (AA^\dagger) \otimes I_n)\vvec(\Psi) & = \vvec(BAA^\dagger - B).
        \end{aligned}
    \end{equation*}

    As $I_m \otimes I_n - (AA^\dagger) \otimes I_n$ is an orthogonal projection,  $\Psi = BAA^\dagger - B$ solves this equation.
    
Replacing the value of $B$ in both expressions found, we have  $\Lambda = \frac{1}{\lambda}A^\dagger{A^\dagger}^\top (V^k-H^{k+1/2})AA^\dagger$ and $\Psi = \frac{1}{\lambda}(V^k-H^{k+1/2})AA^\dagger - \frac{1}{\lambda}(V^k-H^{k+1/2})$.
    \qed\end{proof}

Finally, plugging $\hat\Lambda$ and $\hat\Psi$ into the objective of \eqref{prob_123:dual_variable_ls}, we obtain the following expression for the dual residual for \ref{p1plsp2lin} at iteration $k$ of the DRS algorithm:
\begin{equation*}
    \begin{aligned}
        r_d^{k}  &:=  \textstyle \frac{1}{\lambda}(V^k-H^{k+1/2}) + \frac{1}{\lambda}A^\top AA^\dagger{A^\dagger}^\top (H^{k+1/2}-V^k)AA^\dagger  \\ 
        & \quad \qquad \textstyle + \frac{1}{\lambda}((H^{k+1/2}-V^k)AA^\dagger - (H^{k+1/2}-V^k)) AA^\dagger \\
        &\quad\qquad - \textstyle\frac{1}{\lambda}((H^{k+1/2}-V^k)AA^\dagger - (H^{k+1/2}-V^k)) \\
        & ~=  \textstyle\frac{2}{\lambda}(V^k-H^{k+1/2}) + \frac{1}{\lambda}A^\top {A^\dagger}^\top (H^{k+1/2}-V^k)AA^\dagger + \frac{1}{\lambda}(V^k-H^{k+1/2})AA^\dagger  \\
        &  ~= \textstyle\frac{2}{\lambda}(V^k-H^{k+1/2}) + \frac{1}{\lambda}A^\dagger A (H^{k+1/2}-V^k)AA^\dagger - \frac{1}{\lambda}(H^{k+1/2}-V^k)AA^\dagger. 
    \end{aligned} 
\end{equation*}


\subsection{Numerical experiments}\label{sec:numexp123}

In this section, we compare our different  methods  for computing sparse universal least-squares solvers with minimum rank. We use the same set of test instances used in \S\ref{sec:numexp13}, but here, in addition to using  \texttt{Gurobi}, DRS$_{\mbox{\tiny res}}$, and DRS$_{\mbox{\tiny fp}}$\,, we use the ADMM 
(and its implementation) from \cite{ponte2024goodfastrowsparseahsymmetric} (see \S\ref{sec:admm123}). For the DRS algorithms, the projection $\Pi_{\mathcal{C}}(\cdot)$ and the primal and dual residuals are specified in \S\ref{sec:proj_p123}. We set the parameter $\lambda:= 10^{-2}$ in the DRS algorithms and the parameter $\rho:=3$ in ADMM.

The results for \texttt{Gurobi} were obtained by solving \ref{calP123}\,; this formulation led to the best performance of \texttt{Gurobi} in the preliminary experiments reported in the Appendix. As for \ref{p131}\,, we observed a large impact of the formulation of problem \ref{p1231a} on the runtime of \texttt{Gurobi}.

Furthermore, as for \ref{p131}\,, DRS$_{\mbox{\tiny fp}}$ was the best-performing method for \ref{p1231a}\,, converging faster than DRS$_{\mbox{\tiny res}}$ and ADMM for solutions with similar 1-norms. Thus, we report here again the results of an experiment in which 
we first solved all instances with DRS$_{\mbox{\tiny fp}}$ and saved $\|H_{\mbox{\tiny fp}}\|_1$\,, the 1-norm of the solutions obtained. Then, we ran DRS$_{\mbox{\tiny res}}$ and ADMM, stopping the algorithm when 
the obtained solution $H$ satisfied $(\|H\|_1 - \|H_{\mbox{\tiny fp}}\|_1)/ \|H_{\mbox{\tiny fp}}\|_1 \le 10^{-5}$.

In Table \ref{tab:stats_p123}\,, we compare the results for \texttt{Gurobi}, DRS$_{\mbox{\tiny res}}$ and ADMM with the results for DRS$_{\mbox{\tiny fp}}$\,.  For DRS$_{\mbox{\tiny fp}}$\,, we show the 0-norm and 1-norm of the solution of each instance. For \texttt{Gurobi}, DRS$_{\mbox{\tiny res}}$ and ADMM, we consider the solution $H$ obtained by each algorithm and compare it with the solution $H_{\mbox{\tiny fp}}$ obtained by DRS$_{\mbox{\tiny fp}}$\,, showing the factor $(\|H\| - \|H_{\mbox{\tiny fp}}\|)/\|H_{\mbox{\tiny fp}}\|$ for the 0-norm. We also show the runtime (in seconds) and the 1-norm factor for \texttt{Gurobi}; for the other two algorithms we show 1-norm factor in the column `time', only if the time limit is reached, otherwise we present the runtime of algorithms (in seconds). 

\begin{table}[ht!]
    \caption{Comparison against DRS$_{\mbox{\tiny fp}}$ for \ref{p1231a}  ($n=0.5m$, $r=0.25m$)}
    \label{tab:stats_p123}
\begin{footnotesize}
    \centering  
        \begin{tabular}{r|rrr|rrr|rr|rr}      
         \multicolumn{1}{c}{} & \multicolumn{3}{c|}{\texttt{Gurobi}} & \multicolumn{3}{c|}{DRS$_{\mbox{\tiny fp}}$} & \multicolumn{2}{c|}{DRS$_{\mbox{\tiny res}}$} & \multicolumn{2}{c}{ADMM} \\
        \hline       
        \multicolumn{1}{c|}{$m$} & \multicolumn{1}{c}{$\jatop{\vphantom{\mbox{$X^{X^X}$}} \|H\|_0}{\mbox{factor}}$} & \multicolumn{1}{c}{$\jatop{\|H\|_1}{\mbox{factor}}$} & \multicolumn{1}{c|}{time} & \multicolumn{1}{c}{$\|H\|_0$} & \multicolumn{1}{c}{$\|H\|_1$} & \multicolumn{1}{c|}{time} & \multicolumn{1}{c}{$\jatop{\|H\|_0}{\mbox{factor}}$} & \multicolumn{1}{c|}{time} & \multicolumn{1}{c}{$\jatop{\|H\|_0}{\mbox{factor}}$} & \multicolumn{1}{c}
        {
        $\katop{\mbox{time}}
        {\left(\!{
        \jatop{\|H\|_1}{\mbox{factor}}}\!\right)
        }
        $
        }\\[4pt]
        \hline
           \vphantom{\mbox{$X^{X^X}$}} 
        100  & -6.13e-2 & -3.87e-5 & 9.11 & 3978 & 194.28 & 0.16 & 7.29e-3 & 0.16 & 9.30e-3 & 0.19 \\
        200  & -5.49e-2 & -5.20e-5 & 154.68 & 15811 & 539.73 & 0.86 & 7.46e-3 & 2.81 & 2.03e-2 & 2.62 \\
        300  & -5.01e-2 & -6.74e-5 & 3845.09 & 35540 & 868.25 & 1.77 & 7.12e-3 & 6.50 & 2.18e-2 & 8.30 \\
        400  & \multicolumn{1}{c}{$*$} & \multicolumn{1}{c}{$*$} & \multicolumn{1}{c|}{$*$} & 67754 & 1340.00 & 2.88 & 5.02e-3 & 11.07 & 2.18e-2 & 14.49 \\
        500  & \multicolumn{1}{c}{$*$} & \multicolumn{1}{c}{$*$} & \multicolumn{1}{c|}{$*$} & 106042 & 1870.79 & 4.33 & 6.12e-3 & 20.95 & 2.20e-2 & 22.66 \\
        \hline
           \vphantom{\mbox{$X^{X^X}$}} 
        1000 & \multicolumn{1}{c}{$*$} & \multicolumn{1}{c}{$*$} & \multicolumn{1}{c|}{$*$} & 426933 & 4887.30 & 22.01 & 4.54e-3 & 72.95 & 2.22e-2 & 199.64 \\
        2000 & \multicolumn{1}{c}{$*$} & \multicolumn{1}{c}{$*$} & \multicolumn{1}{c|}{$*$} & 1670089 & 11783.44 & 99.06 & 3.61e-3 & 309.08 & 2.46e-2 & 1106.59 \\
        3000 & \multicolumn{1}{c}{$*$} & \multicolumn{1}{c}{$*$} & \multicolumn{1}{c|}{$*$} & 3766656 & 20462.63 & 202.50& 3.02e-3 & 586.85 & 2.51e-2 & 2482.79 \\
        4000 & \multicolumn{1}{c}{$*$} & \multicolumn{1}{c}{$*$} & \multicolumn{1}{c|}{$*$} & 6597660 & 30032.65 & 399.07 & 2.53e-3 & 1050.99 & 2.41e-2 & 4739.21 \\
        5000 & \multicolumn{1}{c}{$*$} & \multicolumn{1}{c}{$*$} & \multicolumn{1}{c|}{$*$} & 10291141 & 40718.30 & 538.69 & 2.19e-3 & 1471.85 & 2.36e-2 & (1.52e-5) \\
    \end{tabular}
    \end{footnotesize}
\end{table}

Comparing Table \ref{tab:stats_p123} to Table \ref{tab:stats_p13}, we see that solving \ref{p1231a} is more time consuming than solving \ref{p131} for both \texttt{Gurobi} and DRS, but the comparison analysis among the different methods to solve both problems is very similar, except for the fact that we added ADMM for \ref{p1231a}\,, which had worse performance than both DRS algorithms. Once more, we see that \texttt{Gurobi} is not competitive in solving the problem, and that computing the residuals is a big burden for  DRS$_{\mbox{\tiny res}}$\,, which obtains worse 0-norm solutions and takes much longer to converge than DRS$_{\mbox{\tiny fp}}$\,. ADMM took longer than both DRS methods to converge, did not converge in the time limit for the largest instance, and obtained solutions with worse 0-norm than both DRS algorithms for all instances.

As we also observed for \ref{p131}\,, solutions with smaller 1-norms are sparser, showing again the effective use of the 1-norm minimization to induce sparsity. 

In Table \ref{tab:stats2_p123} we compare the solution $H_{\mbox{\tiny fp}}$ obtained with DRS$_{\mbox{\tiny fp}}$ with the Moore-Penrose pseudoinverse of $A$, showing the factors $(\|H_{\mbox{\tiny fp}}\|- \|A^\dagger\|)/\|A^\dagger\|$ for the 0-norm and 1-norm. For \texttt{Gurobi} and  DRS$_{\mbox{\tiny fp}}$\,, we also  show the ratio between the 0-norm of the solution obtained and the upper bound $\beta_{123}:=mr + (m-r)(n-r)$ on the number of nonzero elements of extreme solutions of the standard linear-programming reformulation of \ref{p1231a}
(see Theorem \ref{prop:basicguarantee2}). 

\begin{table}[ht!]
    \centering 
    \caption{More statistics for DRS$_{\mbox{\tiny fp}}$ for \ref{p1231a} ($n=0.5m$, $r=0.25m$)}
    \label{tab:stats2_p123}
\begin{footnotesize} 
        \begin{tabular}{r|r|rrr}        
         \multicolumn{1}{c}{} & \multicolumn{1}{c|}{\texttt{Gurobi}} & \multicolumn{3}{c}{DRS$_{\mbox{\tiny fp}}$} \\ 
        \hline    
        \multicolumn{1}{c|}{$m$} & $\frac{\vphantom{\mbox{$X^{X^X}$}} \|H\|_0}{\beta_{123}}$ & $\frac{\|H\|_0}{\beta_{123}}$ & $\jatop{\|H\|_0}{\mbox{factor}}$ & $\jatop{\|H\|_1}{\mbox{factor}}$  \\ [4pt]
        \hline
\vphantom{\mbox{$X^{X^X}$}} 
        100 & 0.853 & 0.909 & -0.164 & -0.129  \\
        200 & 0.854 & 0.903 & -0.160 & -0.114  \\
        300 & 0.857 & 0.903 & -0.157 & -0.134  \\
        400 & \multicolumn{1}{c|}{$*$} & 0.968 & -0.101 & -0.123  \\
        500 & \multicolumn{1}{c|}{$*$} & 0.970 & -0.096 & -0.119  \\
        \hline
           \vphantom{\mbox{$X^{X^X}$}} 
        1000 & \multicolumn{1}{c|}{$*$} & 0.976 & -0.082 & -0.124 \\
        2000 & \multicolumn{1}{c|}{$*$} & 0.954 & -0.092 & -0.142  \\
        3000 & \multicolumn{1}{c|}{$*$} & 0.957 & -0.078 & -0.143  \\
        4000 & \multicolumn{1}{c|}{$*$} & 0.943 & -0.087 & -0.142  \\
        5000 & \multicolumn{1}{c|}{$*$} & 0.941 & -0.082 & -0.146 \\
    \end{tabular}
    \end{footnotesize}
\end{table}

Comparing  Table \ref{tab:stats2_p123} to Table \ref{tab:stats2_p13}, we first notice that the solutions obtained by \texttt{Gurobi} and DRS$_{\mbox{\tiny fp}}$ have a number of nonzero elements somewhat farther from the upper bound $\beta_{123}$\,, indicating that the upper bound may not be sharp; see additional comments in \S\ref{sec:outlook}. We see that when imposing minimum rank on the ah-symmetric generalized inverses, the 0-norm factors comparing the sparsity of the solutions to the sparsity of $A^\dagger$ increases a lot, showing the important trade-off between low rank and sparse ah-symmetric generalized inverses. On the other hand, the comparisons between 1-norms is not as affected by the low-rank restriction, and we still see decreases of about 13\% in the 1-norm of the solutions when compared to the 1-norm of $A^\dagger$.

In Figure \ref{fig:p123}, we plot the pairs $(\|H\|,\mbox{time (in seconds)})$ every 50 iterations for DRS$_{\mbox{\tiny fp}}$ and ADMM, for both 1-norm and 0-norm, when applied to our largest instance ($m=5000$). In this final experiment for \ref{p1231a}\,, we let both algorithms run for $7200$ seconds to better compare their convergence behavior. We only show results up to $2000$ seconds in the plots, because after this point the norms of the solutions do not change significantly. We clearly see the faster convergence of DRS$_{\mbox{\tiny fp}}$ compared to ADMM. It is interesting to note that after the long run of $7200$ seconds, the 1-norm of the DRS$_{\mbox{\tiny fp}}$ solution decreased by only a factor of $10^{-4}$ compared to the solution reported in Table~\ref{tab:stats_p123}. On the other hand, a solution with the 0-norm decreased by a more significant factor of $10^{-2}$ was obtained after $1200$ seconds. Although we observe a small increase in the 0-norm after this point, we see that the 1-norm was a good surrogate for inducing sparsity.

\begin{figure}[!ht]
    \centering
    \includegraphics[width=0.496\linewidth]{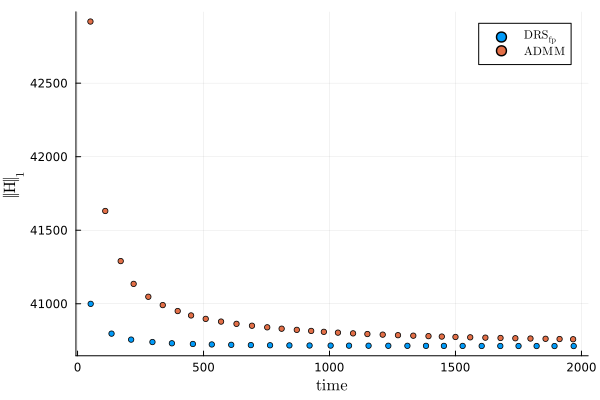}
     \includegraphics[width=0.496\linewidth]{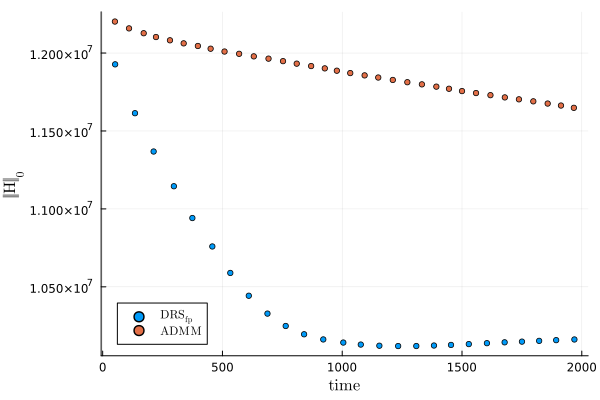}
    \caption{Iterations for the largest instance ($m=5000$) of \ref{p1231a} (stride of 50)}
    \label{fig:p123}
\end{figure}

\section{Sparse simultaneous universal solvers for least-squares and minimum 2-norm }\label{sec:ah-ha}

\subsection{Sparse universal minimum 2-norm solvers
}\label{sec:MN}

A \emph{universal minimum 2-norm solver} for a matrix $A\in\mathbb{R}^{m \times n}$ is a
matrix $H\in\mathbb{R}^{n \times m}$ such that 
$\hat{\theta}:=Hb$ is a solution to the minimum 2-norm problem 
\begin{align}
\min \left\{\|\theta\|^2_2 : A\theta=b \right\}, \label{m2n}\tag{M2N}
\end{align}
for every $b\in\mathcal{R}(A)$.

It is a somewhat-involved calculus exercise to verify the following well-known result.
\begin{theorem}[see e.g. 
\protect{\cite[Thm. 2 of Chap. 3]{BenIsrael1974}}]\label{prop:p1p4}
$H$ satisfies {\rm\ref{P1}+\ref{P4}} (i.e., $H$ is a ha-symmetric generalized inverse of $A$) if and only if $H$ is a universal minimum 2-norm solver for $A$.
\end{theorem}

It is trivial to see the following.
\begin{proposition}\label{trivial} \vphantom{\strut}
\begin{itemize}
    \item[(i)] If {\rm\ref{P1}} holds, then $A^\top H^\top A^\top=A^\top$ (that is,  {\rm\ref{P1}} holds when $A$ is
    replaced by $A^\top$ and  $H$ is
    replaced by $H^\top$).
    \item[(ii)] If {\rm\ref{P2}} holds, then $H^\top A^\top H^\top=H^\top$ (that is,  {\rm\ref{P2}} holds when $A$ is
    replaced by $A^\top$ and  $H$ is
    replaced by $H^\top$).
    \item[(iii)] If {\rm\ref{P3}} holds, then $H^\top A^\top = (H^\top A^\top)^\top$ (that is,  {\rm\ref{P4}} holds when $A$ is
    replaced by $A^\top$ and  $H$ is
    replaced by $H^\top$).
    \item[(iv)] If {\rm\ref{P4}} holds, then $A^\top H^\top = (A^\top H^\top)^\top$ (that is, {\rm\ref{P3}} holds when $A$ is
    replaced by $A^\top$ and  $H$ is
    replaced by $H^\top$).
\end{itemize}    
\end{proposition}

From this and Theorems \ref{thm:proj_p1p3_Rhode}    
and \ref{lem:P1P3eqPLS}, we get the following two results.

\begin{theorem} [see e.g. 
\protect{\cite[Thm. 4 of Chap. 2]{BenIsrael1974}}]\label{thm:proj_p1p4}    
    $H$ satisfies \rm{\ref{P1}+\ref{P4}} if and only if $HA = A^{\dagger}A$, the orthogonal projector onto the range of $A^\top$.
\end{theorem}

\begin{theorem}\label{lem:P1P4eqPMN}  $H$ satisfies \rm{\ref{P1}+\ref{P4}} if and only if $H$ satisfies 
\begin{align}
& AA^\top H^\top  = A. \tag{PMN}\label{PMN} 
\end{align}
\end{theorem}

Because of Proposition \ref{trivial}, we do not explore models/algorithms/experiments for 
1-norm minimizing solutions of \ref{P1}+\ref{P4} and \ref{P1}+\ref{P2}+\ref{P4} 
as these problems are equivalent to working with \ref{P1}+\ref{P3} and \ref{P1}+\ref{P2}+\ref{P3}
with respect to $A^\top$.
Nonetheless, we are interested in \ref{P1}+\ref{P3} because it is fundamental to working
with \ref{P1}+\ref{P3}+\ref{P4}, which, owing to Theorems \ref{lem:P1P3eqPLS} and \ref{lem:P1P4eqPMN}, characterizes the simultaneous universal least-squares and min-norm solvers.
It turns out that we can transpose \ref{PLS} and sum it with \ref{PMN} to obtain a reduced system of equations that is equivalent to imposing both of them.

\begin{theorem}\label{lem:PMX}
    $H$ satisfies \rm{\ref{P1}+\ref{P3}+\ref{P4}} (i.e., $H$ is
    a reflexive ah-ha-symmetric generalized inverse of $A$) if and only if $H$ satisfies 
    \begin{align}
&  AA^\top H^\top  + H^\top A^\top A = 2A.\tag{PMX}\label{PMX} 
\end{align}
\end{theorem}

\begin{proof}\phantom{.}

\noindent     1. $(\Rightarrow)$ If $H$ satisfies \ref{P1}+\ref{P3}+\ref{P4}, then $H$ satisfies \ref{PMX}.

    Because $H$ satisfies \ref{P1}, we have
    \begin{equation*}
            AHA + AHA  = 2A \Rightarrow 
            AA^\top H^\top  + H^\top A^\top A  = 2A,
    \end{equation*}
    where we simply use that $H$ satisfies \ref{P4} (first term) and \ref{P3}  (second term).

\noindent     2. $(\Leftarrow)$ If $H$ satisfies \ref{PMX}, then $H$ satisfies \ref{P1}+\ref{P3}+\ref{P4}.

    We consider the full singular-value decomposition $A = U\Sigma V^\top $ and the representation of $H$ given by $H = V\Gamma U^\top $ (see Theorem \ref{thm:structural}). Because $H$ satisfies \ref{PMX}, we have

    \begin{equation*}
        \begin{aligned}
            AA^\top H^\top  + H^\top A^\top A & = 2A &\Rightarrow \\
            (U\Sigma V^\top )(V\Sigma^\top U^\top )(U\Gamma^\top V^\top ) + (U\Gamma^\top V^\top )(V\Sigma^\top U^\top )(U\Sigma V^\top ) & = 2U\Sigma V^\top  &\Rightarrow \\
            U\Sigma \Sigma^\top \Gamma^\top V^\top  + U\Gamma^\top \Sigma^\top \Sigma V^\top  & = 2U\Sigma V^\top  &\Rightarrow \\
            \Sigma \Sigma^\top \Gamma^\top  + \Gamma^\top \Sigma^\top \Sigma & = 2 \Sigma &\Rightarrow \\
            \begin{bmatrix} D & 0 \\ 0 & 0 \end{bmatrix}\begin{bmatrix} D^\top  & 0 \\ 0 & 0 \end{bmatrix}\begin{bmatrix} X^\top  & Z^\top  \\ Y^\top  & W^\top  \end{bmatrix} + \begin{bmatrix} X^\top  & Z^\top  \\ Y^\top  & W^\top  \end{bmatrix}\begin{bmatrix} D^\top  & 0 \\ 0 & 0 \end{bmatrix}\begin{bmatrix} D & 0 \\ 0 & 0 \end{bmatrix} & = 2\begin{bmatrix} D & 0 \\ 0 & 0 \end{bmatrix} &\Rightarrow \\
            \begin{bmatrix} DD^\top X^\top  + X^\top D^\top D & DD^\top Z^\top  \\ Y^\top D^\top D & 0 \end{bmatrix} & = 2 \begin{bmatrix} D & 0 \\ 0 & 0 \end{bmatrix}.
        \end{aligned}
    \end{equation*}
    Thus, we have $DD^\top Z^\top  = 0$ and $Y^\top D^\top D = 0$. Because $D$ is diagonal with non-zero diagonal elements, the only solution is $Z = 0$, $Y = 0$. From $DD^\top X^\top  + X^\top D^\top D = 2D$ we have

    \begin{equation*}
        \begin{array}{rclr}
        X_{ij}D_{jj}^2+D_{ii}^2X_{ij}&=&\left\{\begin{array}{ll} 2D_{ii}\,,&\mbox{if $i=j$;}\\
        0,&\mbox{otherwise}.
        \end{array}\right.
         \end{array}
    \end{equation*}
    Then, we have $X_{ij} = 0$ for all $i \neq j$, and $X_{ii} = \frac{1}{D_{ii}}$ for all $i = j$. So, $X = D^{-1}$. Hence, by Theorem \ref{thm:structural}, $H$ satisfies \ref{P1}. Because $H$ satisfies \ref{P1}, $Y = 0$, and $Z = 0$, we have again by Theorem \ref{thm:structural} that $H$ satisfies \ref{P3} and \ref{P4}.
   \qed
\end{proof}

The problem of finding a 1-norm minimizing simultaneous  universal least-squares solver
and universal minimum-2 norm solver 
can be formulated as any of the following linear-programming problems.
\begin{align}
        & \textstyle \min_{H \in \mathbb{R}^{n \times m}}\{\|H\|_1 : \text{\ref{P1}+\ref{P3}+\ref{P4}}\}, \tag{$P_{134}^1$}\label{p1p134}\\ 
        & \textstyle \min_{H \in \mathbb{R}^{n \times m}}\{\|H\|_1 : \text{\ref{PMN}+\ref{P3}}\}, \tag{$P_{\text{PMN3}}^1$}\label{p1pmn3}\\    
         & \textstyle \min_{H \in \mathbb{R}^{n \times m}}\{\|H\|_1 : \text{\ref{PLS}+\ref{PMN}}\}, \tag{$P_{\text{PLSPMN}}^1$}\label{p1plspmn}\\  
         & \textstyle \min_{H \in \mathbb{R}^{n \times m}}\{\|H\|_1 : (AH=A A^\dagger) +(HA = A^{\dagger}A)\}, \tag{$P_{\mbox{\scriptsize${\mathcal{R}\mathcal{R}^{\scriptscriptstyle\mathsf{T}}}$}}^1$}\label{double}\\  
        & \textstyle \min_{H \in \mathbb{R}^{n \times m}}\{\|H\|_1 : \text{\ref{PMX}}\},  \tag{$P_{\text{PMX}}^1$}\label{p1pmx}\\  
        & \textstyle\min_{W \in \mathbb{R}^{(n-r) \times (m-r)}}\{\|V_1D^{-1}U_1^\top + V_2WU_2^\top\|_1\}.\tag{$\mathcal{P}_{134}^1$} \label{calP134}
\end{align}

Theorems \ref{prop:p1p3} and \ref{prop:p1p4} lead us to the natural formulation \ref{p1p134}\,. 
From there, 
 using Theorem \ref{lem:P1P4eqPMN}, 
 we obtain \ref{p1pmn3}\,.
Returning to \ref{p1p134}\,, and a bit redundantly 
applying Theorems \ref{lem:P1P3eqPLS} and \ref{lem:P1P4eqPMN}, we obtain 
\ref{p1plspmn}\,. The formulation \ref{p1pmx} comes from Theorem \ref{lem:PMX}. The formulation \ref{calP134}
is derived from
Theorem \ref{thm:structural}, setting $X:=D^{-1}$, $Y:=0$ and $Z:=0$ (similar to 
what is done in \cite{PFLX_ORL}).
All six of these formulations can easily be reformulated as LPs. 
Previous computational work, using linear programming, was only for \ref{p1p134}\,;
see \cite{FFL2016}. 

There is no previous sparsity result for 
\ref{p1p134}\,, so we work one out, using the
same scheme as for Theorems \ref{prop:basicguarantee}
and \ref{prop:basicguarantee2}, but with more difficult algebraic reasoning. 

\begin{lemma}[\protect{\cite[From the first paragraph of the proof of Prop. 2.1]{XFLPsiam}}]\label{lemma:basic}
If $\rank(B)=p$, then every extreme solution of
the standard linear-programming reformulation of  
\[
\min\{\|c \circ x\|_1:~ Bx=b,~ x\in\mathbb{R}^n\}
\]
has at most $p$ nonzeros. 
\end{lemma}

\begin{theorem}\label{prop:basicguarantee3}
Suppose that $A\in \mathbb{R}^{m\times n}$ has rank $r$.
 Extreme solutions of the standard linear-programming reformulation of
 \ref{p1p134} 
 have at most
  $mn - (m-r)(n-r)$ 
  nonzeros. 
\end{theorem}

\begin{proof}
    From Theorem \ref{lem:PMX}, \ref{p1p134}  is equivalent to $\min\{\|H\|_1 : A^\top AH + HAA^\top = 2A^\top\}$. Using the vec operator we have that this problem is equivalent to $\min\{\|\vvec(H)\|_1 : ((I_m \otimes A^\top A) +(AA^\top \otimes I_n))\vvec(H) = 2\vvec(A^\top)\}$. To apply Lemma \ref{lemma:basic} to the standard linear-programming reformulation, 
    we want to compute the rank of $(I_m \otimes A^\top A) +(AA^\top \otimes I_n)$. Let $A = U\Sigma V^T$ be the full singular value decomposition of $A$. So, $(I_m \otimes A^\top A) +(AA^\top \otimes I_n) = (I_m \otimes V\Sigma^T\Sigma V^T) + (U\Sigma\Sigma^T U^T \otimes I_n)$. Because multiplication by an orthogonal transformation preserves rank, we have 
    \begin{equation*}
        \begin{aligned}
          &  (U^T \otimes V^T)((I_m \otimes V\Sigma^T\Sigma V^T) + (U\Sigma\Sigma^T U^T \otimes I_n))(U \otimes V)\\
          & \qquad = (U^TU \otimes \Sigma^T\Sigma) + (\Sigma\Sigma^T \otimes V^TV) \\
            & \qquad = (I_m \otimes \Sigma^T\Sigma) + (\Sigma\Sigma^T \otimes I_n).
        \end{aligned}
    \end{equation*}
    Thus, $\rank((I_m \otimes \Sigma^T\Sigma) + (\Sigma\Sigma^T \otimes I_n)) = \rank((I_m \otimes A^\top A) +(AA^\top \otimes I_n))$.
    Using the Rank-Nullity Theorem we have that $\rank((I_m \otimes \Sigma^T\Sigma) + (\Sigma\Sigma^T \otimes I_n)) = mn - \dim(\mathcal{K}((I_m \otimes \Sigma^T\Sigma) + (\Sigma\Sigma^T \otimes I_n)))$.
    
    The quantity $\dim(\mathcal{K}((I_m \otimes \Sigma^T\Sigma) + (\Sigma\Sigma^T \otimes I_n)))$ is the number of zero eigenvalues counting multiplicity. The eigenvalues of $(I_m \otimes \Sigma^T\Sigma) + (\Sigma\Sigma^T \otimes I_n)$ are $\lambda_{i,j} = \sigma_i^2 + \sigma_j^2$, where $\sigma_i^2$ and $\sigma_j^2$ are the singular values of $\Sigma^T\Sigma$ and $\Sigma\Sigma^T$, respectively.
    Thus, $\dim(\mathcal{K}((I_m \otimes \Sigma^T\Sigma) + (\Sigma\Sigma^T \otimes I_n))) = |\{ \sigma_i^2 + \sigma_j^2 = 0 : i = 1, ..., n, j = 1, ..., m \}|$. Because $\sigma_i^2 = 0$ $\forall i > r$ and $\sigma_j^2 = 0$ $\forall j > r$, we have that $\dim(\mathcal{K}((I_m \otimes \Sigma^T\Sigma) + (\Sigma\Sigma^T \otimes I_n))) = (m-r)(n-r)$.
    Thus, $\rank((I_m \otimes \Sigma^T\Sigma) + (\Sigma\Sigma^T \otimes I_n)) = mn - (m-r)(n-r)$, and
    therefore extreme solutions of the LP for \ref{p1p134} have at most that many nonzeros.
    \qed
\end{proof}

\subsection{ADMM for \texorpdfstring{\ref{p1p134}}{P134}}\label{sec:admm134}

We consider 
ADMM to solve \ref{p1p134}\,, using its equivalent formulation \ref{calP134}\,. Because problem \ref{calP134} is very similar to problem \ref{calP123}\,, our goal is to adapt the ADMM 
from \cite{ponte2024goodfastrowsparseahsymmetric}, discussed in \S\ref{sec:admm123} to \ref{calP134}\,, and verify its computational efficiency running numerical experiments.
Initially, introducing a variable $H \in \mathbb{R}^{n \times m}$, we rewrite \ref{calP134} as
\begin{equation}\label{prob:admm_mat_1normb}
\textstyle\min \{
 \|H\|_{1} ~:~  H =  V_1D^{-1}U_1^\top + V_2WU_2^\top
 \}.
\end{equation}
The augmented Lagrangian associated with \eqref{prob:admm_mat_1normb} is
\begin{align*}
    \textstyle \mathcal{L}_\rho(W,H,\Lambda)
     &:=
     \|H\|_{1} \!+\! \textstyle \frac{\rho}{2}\!\left\|V_1D^{-1}U_1^\top \!+\! V_2WU_2^\top \!-\!H \!+\!\Lambda \right\|^2_F - 
      \frac{\rho}{2}\left\|\Lambda\right\|_F^2\,, 
\end{align*}
where $\rho >0$ is the penalty parameter
and  $\Lambda $ is the scaled Lagrange multiplier; that is,  $\Theta := \rho\Lambda$ is the Lagrange multiplier associated to the constraint in \eqref{prob:admm_mat_1normb}. 
We apply 
ADMM
to \ref{calP134}  by iteratively solving, for $k=0,1,\ldots,$
\begin{align}
    &W^{k+1}:=\textstyle\argmin_W ~ \mathcal{L}_\rho(W,H^{k},\Lambda^k),\label{eq:Wmin1normsubpa}\\
    &H^{k+1}:=\textstyle\argmin_H ~ \mathcal{L}_\rho(W^{k+1},H,\Lambda^k),\label{eq:Emin1normsubpab}\\
    &\textstyle\Lambda^{k+1}:=\Lambda^{k} + V_1D^{-1}U_1^\top + V_2 W^{k+1}U_2^\top - H^{k+1}.\nonumber
\end{align}
Following the ideas presented in \S\ref{sec:admm123}, we show how to update $W$, $H$ and~$\Lambda$.

\medskip

\noindent {\bf{Update \texorpdfstring{$W$}{W}:}} 
Consider \eqref{eq:Wmin1normsubpa}. Defining $J:=  H^k -V_1D^{-1}U_1^\top -\Lambda^k$, we solve 
\begin{equation}\label{eq:Wmin1normsubprob}
    \begin{array}{rl}
        W^{k+1}=\textstyle \argmin_{W}  \left\| J - V_2WU_2^\top\right\|^2_F\,.
    \end{array}
\end{equation}
We can easily verify that the solution of \eqref{eq:Wmin1normsubprob} is given by $W^{k+1} =  V_2^\top J U_2$\,.

\medskip

\noindent {\bf{Update \texorpdfstring{$H$}{H}:}} 
Consider  \eqref{eq:Emin1normsubpab}. Defining $Y := V_1D^{-1}U_1^\top + V_2W^{k+1}U_2^\top  +\Lambda^k$, we solve
\begin{align*}
H^{k+1}=\textstyle\argmin_H\left\{ \|H\|_{1} + \frac{\rho}{2}\left\|H -Y  \right\|^2_F\right\},
\end{align*}
which is equivalent to \eqref{eq:Emin1normsubprob} (for a different $Y$) and has its solution presented in \S\ref{sec:admm123}.

\medskip

\noindent
{\bf{Initialization:}}   
Following \cite{ponte2024goodfastrowsparseahsymmetric},  we consider the dual problem of \eqref{prob:admm_mat_1normb} to initialize $\Lambda$, namely
\begin{equation}\label{prob:dualadmm_mat_1norm}
\textstyle\max_{ \Theta}\{\tr(D^{-1}V_1^\top \Theta U_1)~:~V_2^\top \Theta U_2 = 0,~\|\Theta\|_{\infty} \leq 1\}.
\end{equation}
Then we set  $\textstyle \hat \Theta :=\textstyle \frac{1}{\|V_1 U_1^{\top\vphantom{\Sigma^x}}\|_{\infty}} V_1 U_1^\top$\,, which is feasible to \eqref{prob:dualadmm_mat_1norm} due to the orthogonality of $V$,
and  $\Lambda^0:=\hat\Theta/\rho$.
From Theorem~\ref{thm:structural}, we see that  ah-ha symmetric generalized inverses of  $A$ can be written as $V_1D^{-1}U_1^\top + V_2WU_2^\top$\,. We also see that the M-P pseudoinverse $A^\dagger$ can be written as $V_1D^{-1}U_1^\top$, and we recall that $A^\dagger$ is the generalized inverse of $A$ with minimum Frobenius norm. Then, aiming to obtain $W^1=0$ when solving \eqref{eq:Wmin1normsubprob} at the first iteration of the algorithm, and consequently starting the algorithm with a Frobenius-norm minimizing ah-ha symmetric generalized inverse,  we set $H^0 := V_1D^{-1}U_1^\top + \Lambda^0$.

\medskip

\noindent{\bf{Stopping criterion:}} As in \S\ref{sec:admm123},
we select an absolute tolerance $\epsilon^{\mbox{\scriptsize abs}}$ and a
relative tolerance $\epsilon^{\mbox{\scriptsize rel}}$. The algorithm stops at  iteration $k$ if  
\begin{align*}
    &\|r^{k+1}\|_F \leq \epsilon^{\mbox{\scriptsize abs}}\sqrt{nm} 
    + \epsilon^{\mbox{\scriptsize rel}}\max\{\|H^{k+1}\|_F\,, \|V_2W^{k+1}U_2^\top\|_F\,,\|V_1D^{-1}U_1^\top\|_F\},\\
    &\|s^{k+1}\|_F \leq \epsilon^{\mbox{\scriptsize abs}}
    \sqrt{(n-r)r} +    
\epsilon^{\mbox{\scriptsize rel}}\rho \|V_2^\top\Lambda^{k+1}U_2\|_F\,,
\end{align*}
where $r^{k+1}\! := V_1D^{-1}U_1^\top\! +\! V_2W^{k+1}U_2^\top \!-\! H^{k+1}$, $s^{k+1}\!:= \rho V_2^\top(H^{k+1} \!-\! H^{k})U_2$\, are primal and dual residuals.

As was done for \ref{p1231a}\,, we always construct  a feasible solution for \ref{p1p134}\,, considering the result of Theorem \ref{thm:structural}, and taking as output of 
ADMM
the matrix $V_1D^{-1}U_1^\top + V_2 W^{k+1}U_2^\top$, where $W^{k+1}$ was obtained in its last iteration.


\subsection{DRS for \texorpdfstring{\ref{p1p134}}{P134}}\label{sec:DRS134}

We consider applying DRS to the reformulation of \ref{p1p134} given by \ref{p1plspmn}\,. In this case, we have  $\mathcal{C} := \{H : A^\top AH = A^\top, HAA^\top = A^\top\}$. In the following, we address how to compute a projection onto this affine set and how to compute the primal and dual residuals, as we did in \S\ref{sec:proj_p123}. 
\begin{proposition}\label{projP134}
    If $\mathcal{C} := \{H : A^\top AH = A^\top, HAA^\top = A^\top\}$, for $A\in\mathbb{R}^{m\times n}$, then 
    \[
    \Pi_{\mathcal{C}}(V) = V - A^\dagger AV + A^\dagger - VAA^\dagger + A^\dagger AVAA^\dagger.
    \]
\end{proposition}

\begin{proof}
    Let $V \in \mathbb{R}^{n \times m}$ be a given matrix. We want to find the closed-form solution~of
    \begin{equation}\label{prob:proj_p134}
        \textstyle\argmin_H\{\|V-H\|_F^2 : A^\top AH = A^\top, HAA^\top = A^\top \} .\tag{$\text{Proj}_{134}$}
    \end{equation}
   Using Lagrange multipliers $\Lambda,\Psi\in\mathbb{R}^{n\times m}$,
   the Lagrangian of 
\ref{prob:proj_p134} is  
    \begin{equation*}
        \mathcal{L}(H, \Lambda, \Psi) := \|V-H\|_F^2 + \langle \Lambda, A^\top AH - A^\top \rangle + \langle \Psi, HAA^\top - A^\top \rangle.
    \end{equation*}
As \ref{prob:proj_p134} is convex, its primal-dual solution $(H, \Lambda, \Psi)$ satisfies $\nabla_H \mathcal{L}(H, \Lambda, \Psi) = 0$, or equivalently    
\begin{equation}\label{gradzero}
    H  = V -\textstyle\frac{1}{2}(A^\top A\Lambda + \Psi AA^\top).
\end{equation}
From \eqref{gradzero} and $A^\top AH=A^\top$, we have
        \begin{align}
            A^\top & = A^\top AV -\textstyle\frac{1}{2}A^\top A(A^\top A\Lambda + \Psi AA^\top) & \Rightarrow \nonumber\\
             A^\top AA^\top A\Lambda & = 2(A^\top AV - A^\top) - A^\top A\Psi AA^\top & \Rightarrow \nonumber\\
            A^\dagger {A^\dagger}^\top A^\top AA^\top A\Lambda & = 2A^\dagger {A^\dagger}^\top(A^\top AV - A^\top) - A^\dagger {A^\dagger}^\top A^\top A\Psi AA^\top & \Leftrightarrow \nonumber\\
            A^\top A\Lambda & = 2(A^\dagger AV - A^\dagger) - A^\dagger A\Psi AA^\top.& \label{eqforlambda}
        \end{align}
    From \eqref{gradzero}, $HAA^\top = A^\top$, and \eqref{eqforlambda}, we have
    \begin{align*}
            A^\top & = VAA^\top -\textstyle\frac{1}{2}(A^\top A\Lambda + \Psi AA^\top)AA^\top & \Leftrightarrow \\
            \Psi AA^\top AA^\top & = 2(VAA^\top - A^\top) - A^\top A\Lambda AA^\top & \Rightarrow\\
            \Psi AA^\top AA^\top {A^\dagger}^\top A^\dagger & = 2(VAA^\top - A^\top){A^\dagger}^\top A^\dagger \\ &\qquad\!\!\! - A^\top A\Lambda AA^\top {A^\dagger}^\top A^\dagger & \Leftrightarrow \\
            \Psi AA^\top & = 2(VAA^\dagger - A^\dagger) - A^\top A\Lambda AA^\dagger & \Leftrightarrow \\
            \Psi AA^\top & = 2(VAA^\dagger - A^\dagger) \\ &\qquad \!\!\! \!-\! (2(A^\dagger AV \!-\! A^\dagger) \!-\! A^\dagger\! A\Psi AA^\top) AA^\dagger & \Leftrightarrow \\
            \Psi AA^\top & = 2(VAA^\dagger - A^\dagger) \\ &\qquad \!\!\!-\! 2(A^\dagger\! AVAA^\dagger \!-\! A^\dagger)\! + \!A^\dagger\! A\Psi AA^\top & \Leftrightarrow \\
            \Psi AA^\top - A^\dagger A\Psi AA^\top & = 2(VAA^\dagger - A^\dagger AVAA^\dagger) & \Rightarrow \\
            \Psi AA^\top {A^\dagger}^\top - A^\dagger A\Psi AA^\top {A^\dagger}^\top & = 2(VAA^\dagger - A^\dagger AVAA^\dagger){A^\dagger}^\top & \Leftrightarrow \\
            \Psi A - A^\dagger A\Psi A & = 2(VAA^\dagger - A^\dagger AVAA^\dagger){A^\dagger}^\top & \Rightarrow \\
            \Psi AA^\dagger - A^\dagger A\Psi AA^\dagger & = 2(VAA^\dagger - A^\dagger AVAA^\dagger){A^\dagger}^\top A^\dagger & \Leftrightarrow \\
            ((AA^\dagger) \!\otimes\! I_n \!-\! (AA^\dagger) \!\otimes\! (A^\dagger A))\vvec(\Psi) & = 2\vvec((VAA^\dagger \!-\! A^\dagger AVAA^\dagger){A^\dagger}^\top\! A^\dagger).
        \end{align*}
    As $(AA^\dagger) \otimes I_n - (AA^\dagger) \otimes (A^\dagger A)$ is an orthogonal projection,  $\hat\Psi := 2(VAA^\dagger - A^\dagger AVAA^\dagger){A^\dagger}^\top A^\dagger$ solves the above equation. Substituting it into \eqref{eqforlambda} gives
        \begin{align*}
            A^\top A\Lambda & = 2(A^\dagger AV - A^\dagger) - A^\dagger A\Psi AA^\top & \Leftrightarrow   \\
            A^\top A\Lambda & = 2(A^\dagger AV - A^\dagger)\\ &\qquad - 2A^\dagger A(VAA^\dagger - A^\dagger AVAA^\dagger){A^\dagger}^\top A^\dagger AA^\top & \Leftrightarrow \\
            A^\top A\Lambda & = 2(A^\dagger AV - A^\dagger) -2A^\dagger AVAA^\dagger + 2A^\dagger AVAA^\dagger & \Leftrightarrow \\
            A^\top A\Lambda & = 2(A^\dagger AV - A^\dagger) & \Rightarrow \\
            {A^\dagger}^\top A^\top A\Lambda & = 2{A^\dagger}^\top(A^\dagger AV - A^\dagger) & \Leftrightarrow \\
            A\Lambda & = 2{A^\dagger}^\top(A^\dagger AV - A^\dagger) & \Rightarrow \\
            A^\dagger A\Lambda & = 2A^\dagger{A^\dagger}^\top(A^\dagger AV - A^\dagger) & \Leftrightarrow \\
            (I_m \otimes (A^\dagger A))\vvec(\Lambda) & = 2\vvec(A^\dagger{A^\dagger}^\top(A^\dagger AV - A^\dagger)).
        \end{align*}
   As $I_m \otimes (A^\dagger A)$ is an orthogonal projection,  $\hat\Lambda := 2A^\dagger{A^\dagger}^\top(A^\dagger AV - A^\dagger)$ solves the above equation.

     Finally, substituting $\hat\Psi$ and $\hat\Lambda$ into  \eqref{gradzero} gives
        \begin{align*}
            H& = V - A^\top AA^\dagger{A^\dagger}^\top(A^\dagger AV - A^\dagger) - (VAA^\dagger - A^\dagger AVAA^\dagger){A^\dagger}^\top A^\dagger AA^\top \\
            & = V - A^\dagger A(A^\dagger AV - A^\dagger) - (VAA^\dagger - A^\dagger AVAA^\dagger)AA^\dagger \\
            & = V - A^\dagger AV + A^\dagger - VAA^\dagger + A^\dagger AVAA^\dagger.
        \end{align*}
         Furthermore, we can easily verify that the $H$ obtained is feasible to \ref{prob:proj_p134}\,. 
    \qed
\end{proof}

From Theorems \ref{lem:P1P3eqPLS}, \ref{lem:P1P4eqPMN} and \ref{lem:PMX}, we see that the constraints in \ref{p1plspmn} are equivalent to \ref{PMX}.
So, following \cite{Fu_2020}, we  define the primal residual at iteration $k$  of the DRS algorithm as
\begin{equation*}
    r_p^{k} := A^\top AH^{k+1/2} + H^{k+1/2}AA^\top - 2A^\top.
\end{equation*}

To derive the dual residual at iteration $k$, we consider the Lagrangian of  \ref{p1plspmn} and its subdifferential with respect to $H$:
\begin{equation*}
    \mathcal{L}(H, \Lambda, \Psi) := \|H\|_1 + \langle \Lambda, A^\top AH - A^\top \rangle + \langle \Psi, HAA^\top - A^\top \rangle,
\end{equation*}
\begin{equation*}
    \partial_H \mathcal{L}(H, \Lambda, \Psi) = \partial \|H\|_1 + A^\top A\Lambda + \Psi AA^\top.
\end{equation*}
Then, following the same steps of \S\ref{sec:DRS13}, 
we choose $\Lambda, \Psi$ that solves the following least-squares problem:
\begin{equation}\label{prob_134:dual_variable_ls}
    \textstyle \argmin_{\Lambda,\Psi}\|A^\top A\Lambda + \Psi AA^\top - \frac{1}{\lambda}(H^{k+1/2}-V^k)\|_F^2\,.
\end{equation}

\begin{proposition}
    The solution of \eqref{prob_134:dual_variable_ls} is $\textstyle\hat\Lambda := \frac{1}{\lambda}A^\dagger{A^\dagger}^\top (H^{k+1/2}-V^k)$ and $\textstyle \hat\Psi := (\frac{1}{\lambda}(H^{k+1/2}-V^k) - \frac{1}{\lambda}A^\dagger A(H^{k+1/2}-V^k)){A^\dagger}^\top A^\dagger$.
\end{proposition}

\begin{proof}
    Let $L(\Lambda, \Psi) := \|A^\top A\Lambda + \Psi AA^\top - B\|_F^2$, where $B := \frac{1}{\lambda}(H^{k+1/2}-V^k)$. As $L$ is a  convex function, its minimum is attained at $(\hat\Lambda, \hat\Psi)$  that satisfies $\nabla_\Lambda L(\hat\Lambda, \hat\Psi) = 0$ and $\nabla_\Psi L(\hat\Lambda, \hat\Psi) = 0$. Computing first the gradient with respect to $\Lambda$, we have
\begin{equation*} 
        \nabla_\Lambda L(\Lambda, \Psi) = 2A^\top A(A^\top A\Lambda + \Psi AA^\top - B).
    \end{equation*}

    From $\nabla_\Lambda L(\Lambda, \Psi) = 0$, we have
    \begin{equation*}
        \begin{aligned}
            2A^\top A(A^\top A\Lambda + \Psi AA^\top - B) & = 0 & \Leftrightarrow \\
            A^\top AA^\top A\Lambda & = A^\top A(B - \Psi AA^\top) & \Rightarrow \\
            {A^\dagger}^\top A^\top AA^\top A\Lambda & = {A^\dagger}^\top A^\top A(B - \Psi AA^\top) & \Leftrightarrow \\
            AA^\top A\Lambda & = A(B - \Psi AA^\top) & \Rightarrow \\
            A^\dagger AA^\top A\Lambda & = A^\dagger A(B - \Psi AA^\top) & \Leftrightarrow \\
            A^\top A\Lambda & = A^\dagger A(B - \Psi AA^\top). \\
        \end{aligned}
    \end{equation*}

    Next, we will use this  last equation, first to obtain a value for $\hat \Psi$, and then to obtain a value to $\hat\Lambda$. Computing the gradient of $L$ with respect to $\Psi$, we have

    \begin{equation*} 
        \nabla_\Psi L(\Lambda, \Psi) = 2(A^\top A\Lambda + \Psi AA^\top - B)AA^\top.
    \end{equation*}

    From $\nabla_\Psi L(\Lambda, \Psi) = 0$, we have

    \begin{equation*}
        \begin{aligned}
            (A^\top A\Lambda + \Psi AA^\top - B)AA^\top & = 0 & \Leftrightarrow \\
            \Psi AA^\top AA^\top & = (B - A^\top A\Lambda)AA^\top & \Leftrightarrow \\
            \Psi AA^\top AA^\top & = (B - A^\dagger A(B - \Psi AA^\top))AA^\top & \Leftrightarrow \\
            \Psi AA^\top AA^\top & = (B - A^\dagger AB)AA^\top + A^\dagger A\Psi AA^\top AA^\top & \Leftrightarrow \\
            \Psi AA^\top AA^\top - A^\dagger A\Psi AA^\top AA^\top & = (B - A^\dagger AB)AA^\top & \Rightarrow \\
            \Psi AA^\top AA^\top {A^\dagger}^\top - A^\dagger A\Psi AA^\top AA^\top {A^\dagger}^\top & = (B - A^\dagger AB)AA^\top {A^\dagger}^\top & \Leftrightarrow \\
            \Psi AA^\top A - A^\dagger A\Psi AA^\top A & = (B - A^\dagger AB)A & \Rightarrow \\
            \Psi AA^\top AA^\dagger - A^\dagger A\Psi AA^\top AA^\dagger & = (B - A^\dagger AB)AA^\dagger & \Leftrightarrow \\
            \Psi AA^\top - A^\dagger A\Psi AA^\top & = (B - A^\dagger AB)AA^\dagger & \Rightarrow \\
            \Psi AA^\top {A^\dagger}^\top - A^\dagger A\Psi AA^\top {A^\dagger}^\top & = (B - A^\dagger AB)AA^\dagger {A^\dagger}^\top & \Leftrightarrow \\
            \Psi A - A^\dagger A\Psi A & = (B - A^\dagger AB){A^\dagger}^\top & \Rightarrow \\
            \Psi AA^\dagger - A^\dagger A\Psi AA^\dagger & = (B - A^\dagger AB){A^\dagger}^\top A^\dagger & \Leftrightarrow \\
            ((AA^\dagger) \otimes I_n - (AA^\dagger) \otimes (A^\dagger A))\vvec(\Psi) & = \vvec((B - A^\dagger AB){A^\dagger}^\top A^\dagger).
        \end{aligned}
    \end{equation*}

    As $((AA^\dagger) \otimes I_n - (AA^\dagger) \otimes (A^\dagger A))$ is an orthogonal projection,  $\Psi = (B - A^\dagger AB){A^\dagger}^\top A^\dagger$ solves this equation.  Substituting it back in the last equation for $\Lambda$ we have

    \begin{equation*}
        \begin{aligned}
            A^\top A\Lambda & = A^\dagger A(B - \Psi AA^\top) & \Leftrightarrow \\
            A^\top A\Lambda & = A^\dagger A(B - (B - A^\dagger AB){A^\dagger}^\top A^\dagger AA^\top) & \Leftrightarrow \\
            A^\top A\Lambda & = A^\dagger A(B - BAA^\dagger - A^\dagger ABA A^\dagger) & \Leftrightarrow \\
            A^\top A\Lambda & = A^\dagger AB & \Rightarrow \\
            {A^\dagger}^\top A^\top A\Lambda & = {A^\dagger}^\top A^\dagger AB & \Leftrightarrow \\
            A\Lambda & = {A^\dagger}^\top B & \Rightarrow \\
            A^\dagger A\Lambda &= A^\dagger {A^\dagger}^\top B & \Leftrightarrow \\
            (I_m \otimes (A^\dagger A))\vvec(\Lambda) & = \vvec(A^\dagger {A^\dagger}^\top B)
        \end{aligned}
    \end{equation*}

    As $(I_m \otimes (A^\dagger A))$ is an orthogonal projection,  $\Lambda = A^\dagger {A^\dagger}^\top B$ solves this equation.
Replacing the value of $B$ in both expressions found, we have  $\Lambda = \frac{1}{\lambda}A^\dagger{A^\dagger}^\top (H^{k+1/2}-V^k)$ and $\Psi = (\frac{1}{\lambda}(H^{k+1/2}-V^k) - \frac{1}{\lambda}A^\dagger A(H^{k+1/2}-V^k)){A^\dagger}^\top A^\dagger$. \qed
\end{proof}

Finally, evaluating the objective in \eqref{prob_134:dual_variable_ls} at $\hat\Lambda$ and $\hat\Psi$, we obtain the following expression for the dual residual for \ref{p1plspmn}\,,  at iteration $k$ of the DRS algorithm:  
\begin{equation*}
    \begin{aligned}
        r_d^{k} & :=  \textstyle\frac{1}{\lambda}(V^k-H^{k+1/2}) + \frac{1}{\lambda}A^\top AA^\dagger{A^\dagger}^\top (H^{k+1/2}-V^k) 
        \\
        & \qquad\qquad 
        +\textstyle (\frac{1}{\lambda}(H^{k+1/2}-V^k) - \frac{1}{\lambda}A^\dagger A(H^{k+1/2}-V^k)){A^\dagger}^\top A^\dagger AA^\top  \\
        & ~=  \textstyle\frac{1}{\lambda}(V^k-H^{k+1/2}) + \frac{1}{\lambda}A^\dagger A (H^{k+1/2}-V^k) \\
        & \qquad\qquad + \textstyle(\frac{1}{\lambda}(H^{k+1/2}-V^k) - \frac{1}{\lambda}A^\dagger A(H^{k+1/2}-V^k))AA^\dagger .
    \end{aligned}
\end{equation*}

\subsection{Numerical experiments}\label{sec:numexp134}

In this section, we compare our different  methods  for computing sparse generalized inverses that are simultaneously universal solvers for least-squares and minimum 2-norm. 
We use different test instances from those used in the previous sections, due to our motivating problem of having a sparse matrix $H \in \mathbb{R}^{n\times m}$ that simultaneously solves least squares applications, where in general $m\geq n$, and 2-norm minimization, where in general $m\leq n$. Thus, we generate new instances, but in the same way as the \cite{ponte2024goodfastrowsparseahsymmetric} instances used in \S\ref{sec:numexp13} and \S\ref{sec:numexp123} were generated. Specifically, we use the \texttt{MATLAB} function \texttt{sprand} to generate a random dense matrix $A$  of $m\times n$-dimension, with rank $r$ and with the $r$ positive singular values given by the input vector \texttt{rc}, but now we take $m=n$. Following \cite{FLPXjogo,ponte2024goodfastrowsparseahsymmetric}, we select \texttt{rc} as the decreasing vector $M\times(\alpha^{1},\alpha^{2},\ldots,\alpha^{r})$, where $M:=2$ and $\alpha:=(1/M)^{(2/(r+1))}$.

 We use  \texttt{Gurobi}, DRS$_{\mbox{\tiny res}}$, DRS$_{\mbox{\tiny fp}}$\,, and the ADMM 
 that we developed for \ref{p1p134} described in \S\ref{sec:admm134}. For the DRS algorithms, the projection $\Pi_{\mathcal{C}}(\cdot)$ and the primal and dual residuals are specified in \S\ref{sec:DRS134}. We set the parameter $\lambda:= 10^{-2}$ in the DRS algorithms and the parameter $\rho:=3$ in ADMM.

We performed a preliminary experiment solving the six formulations developed for \ref{p1p134}\,. From the results of this experiment, reported in the Appendix, we see that \ref{double} performs better. We therefore used it for the results presented for \texttt{Gurobi} in this section.

As for \ref{p131} and \ref{p1231a}\,, DRS$_{\mbox{\tiny fp}}$ was the best-performing method for \ref{p1p134}\,, converging faster than DRS$_{\mbox{\tiny res}}$ and ADMM for solutions with similar 1-norms. Thus, we report here again the results of an experiment in which 
we first solved all instances with DRS$_{\mbox{\tiny fp}}$ and saved $\|H_{\mbox{\tiny fp}}\|_1$\,, the 1-norm of the solutions obtained. Then, we ran DRS$_{\mbox{\tiny res}}$ and ADMM, stopping the algorithm when 
the obtained solution $H$ satisfied $(\|H\|_1 - \|H_{\mbox{\tiny fp}}\|_1)/ \|H_{\mbox{\tiny fp}}\|_1 \le 10^{-5}$. 

In Table \ref{tab:stats_p134}\,, we compare the results for \texttt{Gurobi}, DRS$_{\mbox{\tiny res}}$ and ADMM with the results for DRS$_{\mbox{\tiny fp}}$\,.  For DRS$_{\mbox{\tiny fp}}$\,, we show the 0-norm and 1-norm of the solution of each instance. For \texttt{Gurobi}, DRS$_{\mbox{\tiny res}}$ and ADMM, we consider the solution $H$ obtained by each algorithm and compare it with the solution $H_{\mbox{\tiny fp}}$ obtained by DRS$_{\mbox{\tiny fp}}$\,, showing the factor $(\|H\| - \|H_{\mbox{\tiny fp}}\|)/\|H_{\mbox{\tiny fp}}\|$ for the 0-norm. We also show the runtime (in seconds) and the 1-norm factor for \texttt{Gurobi}; for the other two algorithms we show 1-norm factors in the column `time', only if the time limit is reached, otherwise we present the runtime of algorithms (in seconds). 

\begin{table}[ht!]
    \centering  
    \caption{Comparison against DRS$_{\mbox{\tiny fp}}$ for \ref{p1p134} ($n=m$, $r=0.25m$)}
    \label{tab:stats_p134}
\begin{footnotesize}
        \begin{tabular}{r|rrr|rrr|rr|rr}      
         \multicolumn{1}{c}{} & \multicolumn{3}{c|}{\texttt{Gurobi}} & \multicolumn{3}{c|}{DRS$_{\mbox{\tiny fp}}$} & \multicolumn{2}{c|}{DRS$_{\mbox{\tiny res}}$} & \multicolumn{2}{c}{ADMM} \\
        \hline       
        \multicolumn{1}{c|}{$m$} & \multicolumn{1}{c}{$\jatop{\vphantom{\mbox{$X^{X^X}$}} \|H\|_0}{\mbox{factor}}$} & \multicolumn{1}{c}{$\jatop{\|H\|_1}{\mbox{factor}}$} & \multicolumn{1}{c|}{time} & \multicolumn{1}{c}{$\|H\|_0$} & \multicolumn{1}{c}{$\|H\|_1$} & \multicolumn{1}{c|}{time} & \multicolumn{1}{c}{$\jatop{\|H\|_0}{\mbox{factor}}$} & \multicolumn{1}{c|}{time} & \multicolumn{1}{c}{$\jatop{\|H\|_0}{\mbox{factor}}$} & \multicolumn{1}{c}
        {
        $\katop{\mbox{time}}
        {\left(\!{
        \jatop{\|H\|_1}{\mbox{factor}}}\!\right)
        }
        $
        }\\[4pt]
        \hline
           \vphantom{\mbox{$X^{X^X}$}} 
        100 & -1.06e-1 & -1.03e-4 & 280.41 & 4766 & 246.30 & 0.60 & 1.05e-2 & 0.80 & 7.41e-2 & 1.32 \\
        200 & \multicolumn{1}{c}{$*$} & \multicolumn{1}{c}{$*$} & \multicolumn{1}{c|}{$*$} & 19193 & 630.22 & 3.32 & 8.81e-3 & 9.82 & 3.72e-2 & 9.04 \\
        300 & \multicolumn{1}{c}{$*$} & \multicolumn{1}{c}{$*$} & \multicolumn{1}{c|}{$*$} & 41975 & 1029.83 & 8.06 & 1.00e-2 & 27.07 & 3.60e-2 & 19.98 \\
        400 & \multicolumn{1}{c}{$*$} & \multicolumn{1}{c}{$*$} & \multicolumn{1}{c|}{$*$} & 74832 & 1647.64 & 20.03 & 7.86e-3 & 47.43 & 3.52e-2 & 68.72 \\
        500 & \multicolumn{1}{c}{$*$} & \multicolumn{1}{c}{$*$} & \multicolumn{1}{c|}{$*$} & 113753 & 2027.75 & 26.89 & 4.96e-3 & 72.46 & 4.68e-2 & 149.27 \\
        \hline 
           \vphantom{\mbox{$X^{X^X}$}} 
        1000 & \multicolumn{1}{c}{$*$} & \multicolumn{1}{c}{$*$} & \multicolumn{1}{c|}{$*$} & 439108 & 5135.77 & 108.27 & 3.34e-3 & 262.21 & 3.74e-2 & 676.21 \\
        2000 & \multicolumn{1}{c}{$*$} & \multicolumn{1}{c}{$*$} & \multicolumn{1}{c|}{$*$} & 1703138 & 13259.53 & 338.36 & 1.40e-3 & 863.07 & 2.54e-2 & 3227.56 \\
        3000 & \multicolumn{1}{c}{$*$} & \multicolumn{1}{c}{$*$} & \multicolumn{1}{c|}{$*$} & 3790186 & 23487.12 & 684.41 & 8.97e-4 & 1621.44 & 2.24e-2 & (2.16e-5) \\
        4000 & \multicolumn{1}{c}{$*$} & \multicolumn{1}{c}{$*$} & \multicolumn{1}{c|}{$*$} & 6658442 & 34795.47 & 1231.34 & 5.78e-4 & 3000.75 & 2.93e-1 & (8.43e-4) \\
        5000 & \multicolumn{1}{c}{$*$} & \multicolumn{1}{c}{$*$} & \multicolumn{1}{c|}{$*$} & 10261421 & 44757.02 & 2110.16 & 2.18e-4 & 5185.74 & 7.28e-1 & (2.61e-3) \\
    \end{tabular}
    \end{footnotesize}
\end{table}

Our analysis for \ref{p1p134} is similar to the one we had for \ref{p1231a}\,. From the results in Table \ref{tab:stats_p134}, we see that \texttt{Gurobi} is not competitive in solving \ref{p1p134}\,, and that computing the residuals at each iteration in DRS$_{\mbox{\tiny res}}$ makes its convergence slower than DRS$_{\mbox{\tiny fp}}$\,. ADMM took longer to converge than the DRS algorithms, did not converge in the time limit for the three largest instances, and obtained worse 0-norm solutions than both DRS algorithms for all instances.

As we also observed for \ref{p131}\,, solutions with smaller 1-norms are sparser, showing  the effective use of the 1-norm minimization to induce sparsity. 

In Table \ref{tab:stats2_p134}, we compare the solution $H_{\mbox{\tiny fp}}$ obtained by DRS$_{\mbox{\tiny fp}}$ with the Moore-Penrose pseudoinverse of $A$, showing the factors $(\|H_{\mbox{\tiny fp}}\|- \|A^\dagger\|)/\|A^\dagger\|$ for the 0-norm and 1-norm. For \texttt{Gurobi} and  DRS$_{\mbox{\tiny fp}}$\,, we also  show the ratio between the 0-norm of the solution obtained and the upper bound $\beta_{134}:=mn + (m-r)(n-r)$ on the number of nonzero elements of extreme solutions of the standard linear-programming reformulation of \ref{p1p134}
(see Theorem \ref{prop:basicguarantee3}). 

\begin{table}[ht!]
    \centering  
    \caption{More statistics for DRS$_{\mbox{\tiny fp}}$ for  \ref{p1p134} ($n=m$, $r=0.25m$)}
    \label{tab:stats2_p134}
\begin{footnotesize}
        \begin{tabular}{r|r|rrrr}       
         \multicolumn{1}{c}{} & \multicolumn{1}{c|}{\texttt{Gurobi}} & \multicolumn{4}{c}{DRS$_{\mbox{\tiny fp}}$} \\ 
        \hline    
        \multicolumn{1}{c|}{$m$} & $\frac{\vphantom{\mbox{$X^{X^X}$}} \|H\|_0}{\beta_{134}}$ & $\frac{\|H\|_0}{\beta_{134}}$ & $\jatop{\|H\|_0}{\mbox{factor}}$ & $\jatop{\|H\|_1}{\mbox{factor}}$ & $\frac{\text{rank}(H)}{m}$ \\ [4pt]
        \hline
           \vphantom{\mbox{$X^{X^X}$}} 
        100 & 0.974 & 1.089 & -0.484 & -0.113 & 0.980 \\
        200 & \multicolumn{1}{c|}{$*$} & 1.097 & -0.489 & -0.122 & 0.990 \\
        300 & \multicolumn{1}{c|}{$*$} & 1.066 & -0.500 & -0.117 & 0.990 \\
        400 & \multicolumn{1}{c|}{$*$} & 1.069 & -0.496 & -0.117 & 0.988 \\
        500 & \multicolumn{1}{c|}{$*$} & 1.040 & -0.502 & -0.111 & 0.988 \\
                \hline    \vphantom{\mbox{$X^{X^X}$}} 
        1000 & \multicolumn{1}{c|}{$*$} & 1.004 & -0.510 & -0.112 & 0.994 \\
        2000 & \multicolumn{1}{c|}{$*$} & 0.973 & -0.517 & -0.113 & 0.996 \\
        3000 & \multicolumn{1}{c|}{$*$} & 0.963 & -0.514 & -0.111 & 0.997 \\
        4000 & \multicolumn{1}{c|}{$*$} & 0.951 & -0.516 & -0.112 & 0.996 \\
        5000 & \multicolumn{1}{c|}{$*$} & 0.938 & -0.508 & -0.108 & 0.998 \\
    \end{tabular}
    \end{footnotesize}
\end{table}

We can observe from the results in Table \ref{tab:stats2_p134} that by minimizing the 1-norm, we obtain ah-ha symmetric generalized inverses with the 0-norm reduced by about 50\% and the 1-norm reduced by almost 90\%, compared to $A^\dagger$. We see that the number of nonzero elements in the \texttt{Gurobi} solution is about 97\% of the upper bound for extreme solutions derived in Theorem \ref{prop:basicguarantee3}, indicating that the upper bound may be sharp; see additional comments in \S\ref{sec:outlook}. 
The number of nonzero elements in the DRS$_{\mbox{\tiny fp}}$ solutions is also not far from this bound. As we observed for \ref{p131}\,, we have larger ratios for the smallest instances, but we note that the ratios might decrease if we let DRS$_{\mbox{\tiny fp}}$ run longer, giving up on its faster convergence. Finally, we see that to obtain the sparse solutions, we give up on low rank; all solutions are close to full-rank matrices.

\begin{figure}[!ht]
    \centering
    \includegraphics[width=0.496\linewidth]{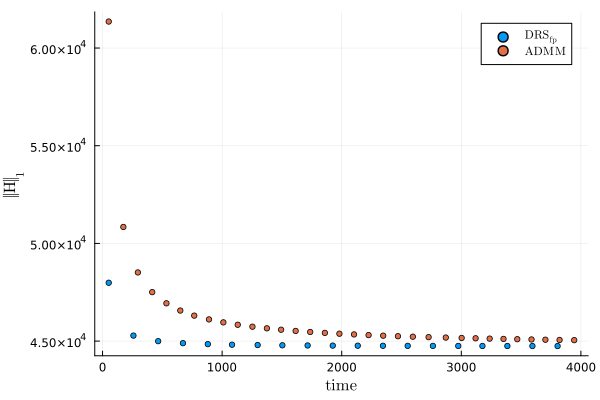}
    \includegraphics[width=0.496\linewidth]{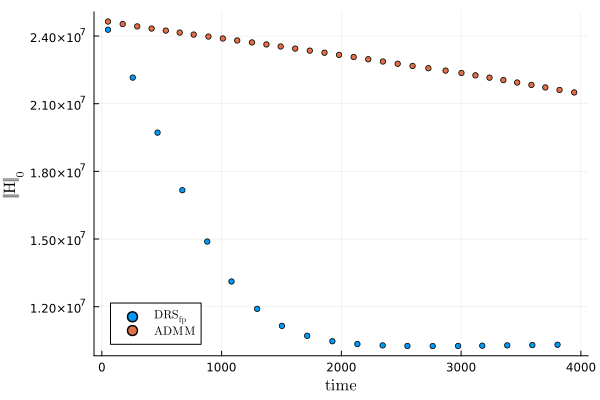}
    \caption{Iterations for \ref{p1p134} (stride of 50)}
    \label{fig:p134}
\end{figure}

In Figure \ref{fig:p134}, we plot the pairs $(\|H\|,\mbox{time (in seconds)})$ every 50 iterations for DRS$_{\mbox{\tiny fp}}$ and ADMM, for both 1-norm and 0-norm, when applied to the largest instance ($m=5000$). We let both algorithms run for $7200$ seconds and show results up to $4000$ seconds in the plots. After this point, the 1-norms of the solutions do not change significantly. As we saw for \ref{p1231a}, we observe the faster convergence of DRS$_{\mbox{\tiny fp}}$ compared to ADMM. We also observe that after $7200$ seconds, the 1-norm of the DRS$_{\mbox{\tiny fp}}$ solution decreased by a factor of only $10^{-4}$ compared to the solution reported in Table~\ref{tab:stats_p134}. The solution with the minimum 0-norm was obtained after $2700$ seconds, but it also decreased by a factor of only $10^{-4}$ compared to the solution reported in Table~\ref{tab:stats_p134}. Once more, we note that although there is a small increase in the 0-norm after $2700$ seconds, we see that the 1-norm was a good substitute for inducing sparsity.

\section{Outlook}\label{sec:outlook}

Overall, our best-performing methods were DRS algorithms. 
 As considered in \cite{Fu_2020}, sometimes the performance of DRS algorithms can be improved using Anderson Acceleration. The challenge, of course, is to balance potentially fewer iterations against the per-iteration cost. We leave this for a future investigation. 
 
The sparsity bounds of Theorem \ref{prop:basicguarantee2} (for \ref{p1231a}) and Theorem \ref{prop:basicguarantee3}
(for  \ref{p1p134}) may not be sharp. 
Our bounds are based on exactly calculating the
rank of the constraint matrix for the associated standard linear-programming reformulations.
So, if better bounds exist, they must be based on demonstrating a positive lower bound on the 
number of basic variables at value 0 (i.e., guaranteed degeneracy) over all basic feasible solutions. Based on similar computational results on $\|H\|_0/\beta_{13}$ (Table \ref{tab:stats2_p13}) and
$\|H\|_0/\beta_{134}$  (Table \ref{tab:stats2_p134}), we see some evidence that the sparsity bound of Theorem \ref{prop:basicguarantee3}
(for  \ref{p1p134}) may be sharp (as it is for 
\ref{p131}\,, via Theorem \ref{prop:basicguarantee}),
and we believe that it is worth trying to establish this.
In contrast, we see significantly lower values for $\|H\|_0/\beta_{123}$ (Table \ref{tab:stats2_p123}), which we 
see as some evidence that the sparsity bound of Theorem \ref{prop:basicguarantee2} (for \ref{p1231a}) is likely not to be sharp, and so we believe that an improved sparsity is likely to exist. 

\cite{XFLPsiam} established that both sparsity-maximization problems
$\min\{\| H \|_0 $ $ ~:~ \mbox{{\ref{P1}+\ref{P2}}}\}$
and
$\min\{\| H \|_0 ~:~ \allowbreak \mbox{\ref{P1}+\ref{P2}+\ref{P3}}\}$
are NP-hard, and they mentioned that 
the complexity of 
$\min\{\| H \|_0 ~:~ \allowbreak \mbox{\ref{P1}+\ref{P3}+\ref{P4}}\}$
is unknown. It remains unknown, and we hope that our work on \ref{p1p134}
and its relatives
sparks interest in settling this open problem. 

\appendix

\normalsize

\section{Appendix}\label{appA}

In the following three tables, we present the results of our preliminary numerical experiments, comparing the runtimes for \texttt{Gurobi} applied to all the formulations presented in this paper for each problem addressed, namely \ref{p131}\,, \ref{p1231a}\,, and \ref{p1p134}\,. The instances used in these experiments are smaller than those used in our final numerical experiments,  reported in \S\S \ref{sec:numexp13}, \ref{sec:numexp123}, and \ref{sec:numexp134}\,. These instances were used to evaluate \texttt{Gurobi}'s performance, which could only solve smaller instances within our time limit. We generated them in the same way as the instances used in our final experiments. We used the \texttt{MATLAB} function \texttt{sprand} to generate random dense matrices $A$  of $m\times n$-dimension, with rank $r$ and with the $r$ positive singular values given by the input vector  \texttt{rc}, where  \texttt{rc} was selected as the decreasing vector $M\times(\alpha^{1},\alpha^{2},\ldots,\alpha^{r})$, where $M:=2$ and $\alpha:=(1/M)^{(2/(r+1))}$.
The time limit used in these preliminary experiments was 1200 seconds for each instance. We generated five instances of each size for each problem and present average runtimes for them in the tables. 

In Table \ref{tab:p1p3vspls}, we present results comparing the different reformulations of  problem \ref{p131}\,. We see that  \ref{p131} and \ref{calP13} are solved in a time two orders of magnitude greater than the others. Although \ref{calP13} has fewer variables than the others, it does not lead to a reduction in the time required to solve the problems. Finally, we note that while we can  solve instances up to $m = 480$ with \ref{ppls1}\,, we can only solve instances up to (at most) $m = 260$ with the other formulations.

\begin{table}[ht!]
    \centering  
    \caption{Average runtimes for        \texttt{Gurobi} for  \ref{p131}\,, \ref{p1proj13}\,, \ref{ppls1}\,, \ref{calP13} ($n=0.5m$, $r=0.25m$)}
    \label{tab:p1p3vspls}
\begin{footnotesize}
        \begin{tabular}{c|r|r|r|r}       
         \multicolumn{1}{c|}{} & \multicolumn{4}{c}{Time (sec)} \\      
        \multicolumn{1}{c|}{$m$} & \multicolumn{1}{c|}{\ref{p131}} & \multicolumn{1}{c|}{\ref{p1proj13}}&\multicolumn{1}{c|}{\ref{ppls1}}& \multicolumn{1}{c}{\ref{calP13}}  \\[4pt]
        \hline      
        60 & 16.65 & 0.54 & 0.54 & 8.66 \\
        80 & 74.27 & 0.45 & 0.37 & 28.51 \\
        100 & 266.64 & 1.08 & 0.78 & 75.03 \\
        120 & 712.28 & 2.05 & 1.44 & 219.74 \\
        140 & \multicolumn{1}{c|}{$*$} & 4.27 & 2.49 & 552.42 \\
        160 & \multicolumn{1}{c|}{$*$} & 7.13 & 4.08 & 1063.08 \\
        180 & \multicolumn{1}{c|}{$*$} & 12.39 & 5.87 & \multicolumn{1}{c}{$*$} \\
        200 & \multicolumn{1}{c|}{$*$} & 17.03 & 9.57 & \multicolumn{1}{c}{$*$} \\
        220 & \multicolumn{1}{c|}{$*$} & 23.47 & 13.77 & \multicolumn{1}{c}{$*$} \\
        240 & \multicolumn{1}{c|}{$*$} & 34.53 & 18.84 & \multicolumn{1}{c}{$*$} \\
        260 & \multicolumn{1}{c|}{$*$} & 46.55 & 26.33 & \multicolumn{1}{c}{$*$} \\
        280 & \multicolumn{1}{c|}{$*$} & \multicolumn{1}{c|}{$*$} & 35.18 & \multicolumn{1}{c}{$*$} \\
        300 & \multicolumn{1}{c|}{$*$} & \multicolumn{1}{c|}{$*$} & 42.66 & \multicolumn{1}{c}{$*$} \\
        320 & \multicolumn{1}{c|}{$*$} & \multicolumn{1}{c|}{$*$} & 59.66 & \multicolumn{1}{c}{$*$} \\
        340 & \multicolumn{1}{c|}{$*$} & \multicolumn{1}{c|}{$*$} & 72.68 & \multicolumn{1}{c}{$*$} \\
        360 & \multicolumn{1}{c|}{$*$} & \multicolumn{1}{c|}{$*$} & 86.84 & \multicolumn{1}{c}{$*$} \\
        380 & \multicolumn{1}{c|}{$*$} & \multicolumn{1}{c|}{$*$} & 114.23 & \multicolumn{1}{c}{$*$} \\
        400 & \multicolumn{1}{c|}{$*$} & \multicolumn{1}{c|}{$*$} & 143.03 & \multicolumn{1}{c}{$*$} \\
        420 & \multicolumn{1}{c|}{$*$} & \multicolumn{1}{c|}{$*$} & 170.12 & \multicolumn{1}{c}{$*$} \\
        440 & \multicolumn{1}{c|}{$*$} & \multicolumn{1}{c|}{$*$} & 212.11 & \multicolumn{1}{c}{$*$} \\
        460 & \multicolumn{1}{c|}{$*$} & \multicolumn{1}{c|}{$*$} & 244.67 & \multicolumn{1}{c}{$*$} \\
        480 & \multicolumn{1}{c|}{$*$} & \multicolumn{1}{c|}{$*$} & 340.35 & \multicolumn{1}{c}{$*$} \\
        500 & \multicolumn{1}{c|}{$*$}& \multicolumn{1}{c|}{$*$}& \multicolumn{1}{c|}{$*$}& \multicolumn{1}{c}{$*$}\\[-4pt]
        \multicolumn{1}{c|}{\vdots}& \multicolumn{1}{c|}{\vdots}& \multicolumn{1}{c|}{\vdots}& \multicolumn{1}{c|}{\vdots}& \multicolumn{1}{c}{\vdots}\\
        600 &  \multicolumn{1}{c|}{$*$}& \multicolumn{1}{c|}{$*$}& \multicolumn{1}{c|}{$*$}& \multicolumn{1}{c}{$*$}
    \end{tabular}
    \end{footnotesize}
\end{table}

In Table \ref{tab:p1p2p3vsplsr}, we present results comparing the different reformulations of  problem \ref{p1231a}\,.
 The formulation that performed best  was \ref{calP123}\,, taking 2 orders of magnitude less time than the second best performing formulation \ref{p1plsp2lin}\,. While we can solve instances with $m$ up to $280$ with \ref{calP123}\,, we can only solve instances with $m$ up to (at most) $120$ with all the other formulations. For this problem, we  see the best performance for the formulation with fewer variables, there are only  $(n-r)r$ variables in \ref{calP123}\,, compared to $mn$ variables in the other formulations.

\begin{table}[ht!]
    \centering  
    \caption{Average runtimes for        \texttt{Gurobi} for  \ref{p1231lin}\,, \ref{p1proj13p2lin}\,, \ref{p1plsp2lin}\,, \ref{pplsr1}\,, \ref{calP123} ($n=0.5m$, $r=0.25m$)}
    \label{tab:p1p2p3vsplsr}
\begin{footnotesize}
        \begin{tabular}{c|r|r|r|r|r}        
         \multicolumn{1}{c|}{} & \multicolumn{5}{c}{Time (sec)} \\   
        \multicolumn{1}{c|}{$m$} & \multicolumn{1}{c|}{\ref{p1231lin}} & \multicolumn{1}{c|}{\ref{p1proj13p2lin}}&\multicolumn{1}{c|}{\ref{p1plsp2lin}}& \multicolumn{1}{c|}{\ref{pplsr1}}& \multicolumn{1}{c}{\ref{calP123}}  \\[4pt]
        \hline      
        60 & 41.63 & 5.89 & 4.28 & 16.43 & 1.10 \\
        80 & 217.41 & 18.84 & 22.59 & 80.95 & 2.81 \\
        100 & \multicolumn{1}{c|}{$*$} & \multicolumn{1}{c|}{$*$} & 161.73 & 332.17 & 5.76 \\
        120 & \multicolumn{1}{c|}{$*$} & \multicolumn{1}{c|}{$*$} & 243.45 & \multicolumn{1}{c|}{$*$} & 11.39 \\
        140 & \multicolumn{1}{c|}{$*$} & \multicolumn{1}{c|}{$*$} & \multicolumn{1}{c|}{$*$} & \multicolumn{1}{c|}{$*$} & 21.13 \\
        160 & \multicolumn{1}{c|}{$*$} & \multicolumn{1}{c|}{$*$} & \multicolumn{1}{c|}{$*$} & \multicolumn{1}{c|}{$*$} & 37.98 \\
        180 & \multicolumn{1}{c|}{$*$} & \multicolumn{1}{c|}{$*$} & \multicolumn{1}{c|}{$*$} & \multicolumn{1}{c|}{$*$} & 82.50 \\
        200 & \multicolumn{1}{c|}{$*$} & \multicolumn{1}{c|}{$*$} & \multicolumn{1}{c|}{$*$} & \multicolumn{1}{c|}{$*$} & 142.42 \\
        220 & \multicolumn{1}{c|}{$*$} & \multicolumn{1}{c|}{$*$} & \multicolumn{1}{c|}{$*$} & \multicolumn{1}{c|}{$*$} & 331.32 \\
        240 & \multicolumn{1}{c|}{$*$} & \multicolumn{1}{c|}{$*$} & \multicolumn{1}{c|}{$*$} & \multicolumn{1}{c|}{$*$} & 510.44 \\
        260 & \multicolumn{1}{c|}{$*$} & \multicolumn{1}{c|}{$*$} & \multicolumn{1}{c|}{$*$} & \multicolumn{1}{c|}{$*$} & 3382.58 \\
        280 & \multicolumn{1}{c|}{$*$} & \multicolumn{1}{c|}{$*$} & \multicolumn{1}{c|}{$*$} & \multicolumn{1}{c|}{$*$} & 5048.71 \\
        300 & \multicolumn{1}{c|}{$*$} & \multicolumn{1}{c|}{$*$} & \multicolumn{1}{c|}{$*$} & \multicolumn{1}{c|}{$*$} & \multicolumn{1}{c}{$*$} \\[-4pt]
        \multicolumn{1}{c|}{\vdots} &\multicolumn{1}{c|}{\vdots}& \multicolumn{1}{c|}{\vdots}& \multicolumn{1}{c|}{\vdots}& \multicolumn{1}{c|}{\vdots}& \multicolumn{1}{c}{\vdots}\\
        600 & \multicolumn{1}{c|}{$*$}& \multicolumn{1}{c|}{$*$}& \multicolumn{1}{c|}{$*$}& \multicolumn{1}{c|}{$*$}& \multicolumn{1}{c}{$*$}
    \end{tabular}
    \end{footnotesize}
\end{table}

\medskip

In Table \ref{tab:p1p3p4vspmx}, we present results comparing the different reformulations of problem \ref{p1p134}\,.
We see that only a few instances can be solved; none of the formulations solves instances with $m>100$.  \ref{p1pmn3} and \ref{double} perform best, taking half the time required by \ref{calP134}\,, which is not expected given that \ref{calP134} has fewer variables than the other reformulations. We note that the time for \ref{p1p134} is an order of magnitude greater than for the other formulations, and although  \ref{p1pmx}  has fewer constraints than \ref{p1p134}\,, \ref{p1pmn3}\,, \ref{p1plspmn}\,, and \ref{double}\,, we only solve instances with $m$ up to $80$ with it. 
\begin{table}[ht!]
    \centering  
    \caption{Average runtimes for \texttt{Gurobi} for  \ref{p1p134}\,, \ref{p1pmn3}\,, \ref{p1plspmn}\,, \ref{double}\,, \ref{p1pmx}\,, \ref{calP134} ($n=m$, $r=0.25m$)}
    \label{tab:p1p3p4vspmx}
\begin{footnotesize}
        \begin{tabular}{c|r|r|r|r|r|r}       
         \multicolumn{1}{c|}{} & \multicolumn{6}{c}{Time (sec)} \\      
        \multicolumn{1}{c|}{$m$} & \multicolumn{1}{c|}{\ref{p1p134}} & \multicolumn{1}{c|}{\ref{p1pmn3}}&\multicolumn{1}{c|}{\ref{p1plspmn}}& \multicolumn{1}{c|}{\ref{double}}& \multicolumn{1}{c|}{\ref{p1pmx}} & \multicolumn{1}{c}{\ref{calP134}}  \\[4pt]
        \hline      
        60 & 244.51 & 8.59 & 8.02 & 6.61 & 12.11 & 30.38 \\
        80 & \multicolumn{1}{c|}{$*$} & 113.97 & 70.94 & 62.79 & 135.51 & 300.91 \\
        100 & \multicolumn{1}{c|}{$*$} & 263.09 & 432.78 & 276.47 & \multicolumn{1}{c|}{$*$} & 647.92 \\
        120 & \multicolumn{1}{c|}{$*$} & \multicolumn{1}{c|}{$*$} & \multicolumn{1}{c|}{$*$} & \multicolumn{1}{c|}{$*$} & \multicolumn{1}{c|}{$*$} & \multicolumn{1}{c}{$*$} \\[-4pt]
        \multicolumn{1}{c|}{\vdots}& \multicolumn{1}{c|}{\vdots}& \multicolumn{1}{c|}{\vdots}& \multicolumn{1}{c|}{\vdots}& \multicolumn{1}{c|}{\vdots}& \multicolumn{1}{c|}{\vdots}& \multicolumn{1}{c}{\vdots}\\
        600 &  \multicolumn{1}{c|}{$*$}& \multicolumn{1}{c|}{$*$}& \multicolumn{1}{c|}{$*$}& \multicolumn{1}{c|}{$*$}& \multicolumn{1}{c|}{$*$}& \multicolumn{1}{c}{$*$}\\
    \end{tabular}
    \end{footnotesize}
\end{table}

\section*{Declarations}~

\noindent \underline{Ethics approval and consent to participate}

\medskip

- not applicable - 

\medskip

\noindent \underline{Consent for publication}

\medskip

- not applicable -

\medskip
 
\noindent \underline{Funding}

\medskip

M. Fampa was supported in part by CNPq grant 307167/2022-4.
J. Lee was supported in part by AFOSR grant FA9550-22-1-0172.

\medskip

\noindent  \underline{Availability of data and materials}

\medskip

Computational experiments were conducted with randomly-generated data. We gave details on how to generate 
comparable data. We have made the test data available at \url{https://zenodo.org/records/20708172}.

\medskip

\noindent \underline{Acknowledgements}

\medskip

We thank Gabriel Ponte for his support and assistance.

\bibliographystyle{alpha}
\bibliography{MFL_Ter}

\end{document}